\documentclass[onefignum,onetabnum]{siamart250211}

\usepackage{lipsum}
\usepackage{amsfonts}
\usepackage{graphicx}
\usepackage{epstopdf}
\usepackage{algorithmic}
\ifpdf
  \DeclareGraphicsExtensions{.eps,.pdf,.png,.jpg}
\else
  \DeclareGraphicsExtensions{.eps}
\fi

\newsiamremark{remark}{Remark}
\newsiamremark{hypothesis}{Hypothesis}
\newsiamremark{conjecture}{Conjecture}
\crefname{hypothesis}{Hypothesis}{Hypotheses}
\newsiamthm{claim}{Claim}

\headers{Numerical Conformal Mappings}{H. Hakula, A. Rasila, Y. Zheng}

\title{Efficient Numerical Conformal Mappings on Multiply Connected Riemann Surfaces}

\author{Harri Hakula\thanks{Department of Mathematics and System Analysis, Aalto University, P.O. Box 11100,
FI--00076 Aalto, Finland 
  (\email{Harri.Hakula@aalto.fi}).}
\and Antti Rasila\thanks{Department of Mathematics with Computer Science, Guangdong Technion -- Israel Institute of Technology, Shantou, Guangdong 515063, P.R. of China and Department of Mathematics with Computer Science, Technion -- Israel Institute of Technology
Haifa 32000, Israel
  (\email{antti.rasila@iki.fi}, \email{antti.rasila@gtiit.edu.cn}).}
\and Yufan Zheng\thanks{Department of Mathematics with Computer Science, Guangdong Technion -- Israel Institute of Technology, Shantou, Guangdong 515063, P.R. of China and Department of Mathematics with Computer Science, Technion -- Israel Institute of Technology
Haifa 32000, Israel
  (\email{Yufan.zheng@gtiit.edu.cn}).}
}

\usepackage{amsopn}

\makeatletter
\newcommand*{\addFileDependency}[1]{
  \typeout{(#1)}
  \@addtofilelist{#1}
  \IfFileExists{#1}{}{\typeout{No file #1.}}
}
\makeatother

\usepackage{subfig}
\usepackage{interval}
\intervalconfig{soft open fences}

\usepackage{booktabs}
\usepackage{enumitem}

\newcommand{\Sf}{\mathcal{S}} 

\newcommand{\C}{\mathbb{C}}

\graphicspath{{PDF/}{Graphics/}}

\ifpdf
\hypersetup{
  pdftitle={Efficient Numerical Conformal Mappings on Multiply Connected Riemann Surfaces},
  pdfauthor={H. Hakula, A. Rasila, Y. Zheng}
}
\fi

\begin{document}

\maketitle

\begin{abstract}
  The conjugate function method is an algorithm for numerical computation of conformal mappings for simply and multiply connected domains on surfaces.
  In this paper the conjugate
  function method, earlier used for simply connected domains, is generalized and refined to achieve the same level of accuracy on
  multiply connected planar domains and Riemann surfaces.
  The main challenge is the accurate and efficient construction of
  boundary values for the conjugate problem on multiply connected domains.
  The method relies on high-order finite element methods which allow for 
  highly accurate computations of mappings on surfaces, including domains of complex boundary geometry containing strong singularities and cusps.
  We also derive the reciprocal error estimate for the multiply connected case.
    The efficacy of the proposed method is illustrated via an extensive set of numerical experiments with error estimates.
	\end{abstract}

\begin{keywords}
  Laplace--Beltrami, numerical conformal mappings, conformal modulus,  Riemann surfaces
\end{keywords}

\begin{AMS}
  30C85, 30F10, 31A15, 65E05, 65E10, 65N30
\end{AMS}

\section{Introduction}
\label{sec:intro}

Conformal geometry has many applications such as engineering (e.g. \cite{gonzales}) 
and mathematical geodesy (see \cite{Vermeer-Rasila}). 
Traditionally, numerical methods of approximating conformal mappings have been mostly 
restricted to planar domains. Even in this setting finding
mappings between multiply connected domains has remained challenging.
M.M.S. Nasser has used the boundary integral equation (BIE) with impressive results
\cite{doi:10.1137/070711438,doi:10.1137/120901933}. Furthermore,
in a very recent work, Nasser discusses efficient implementation of  
the generalized Koebe’s iterative method \cite{LeeMuridNasserYeak2025FastKoebe}.

Koebe's iterative method can also be implemented for surfaces. Gu, Luo, and Yau give an
excellent overview of the state of the art in \cite{864afe467f9949cea8cb49d061b1b5d6}.
Due to both theoretical and implementation complexity, there are relatively few
publications addressing specifics. See, for example, \cite{Li2022,Lin2021,Lin2022,Song2014SymmetricConformal,Yueh2020b}.
Even though a properly implemented Koebe's iteration converges reasonably fast (at least in the
planar case), it is very difficult to provide any reliable error estimates or
convergence rates.

Our method of choice for conformal mappings, 
the conjugate function method, was first introduced in \cite{hqr} for simply and doubly connected planar domains. We use the $hp$-adaptive finite element method ($hp$-FEM) introduced in \cite{hrv} to compute potential functions and conformal moduli, but other methods can also be used, e.g., in \cite{hrs} the stochastic version of the conjugate function algorithm was given for circular arc domains in the plane. An advantage of $hp$-FEM is the mature error analysis. 
Interestingly, the reciprocal identity relating the conformal capacities of the primary and conjugate problems can also be viewed as an error indicator.
In this work this concept is made precise for surfaces with multiple boundary components in the specific setting of 
so-called Q-type canonical domains (see \cite{hqr2} and below).

On method development the main contribution of this work is the new, direct algorithm
for the construction of the conjugate problem for multiply connected problems.
Given the finite element discretization of the primary problem, the problem
of finding the right boundary conditions for the conjugate problem, that is,
the necessary Dirichlet-Neumann swap of the boundary conditions, can be reduced
to a quadratic minimization problem that can be solved with a linear system of equations
whose dimension is the number of the holes. Not only is the approach much faster than the
previous ones, but the accuracy of the solution of the conjugate problem is brought to the
same level as that of the primary problem. 
The same algorithm applies to problems on multiply connected domains on surfaces as well without any modifications.
In view of the results in this paper, the conjugate function method is now
comparable to the state-of-the-art methods with the additional benefits of the
high-order finite element frameworks such as advanced error estimation and exponential convergence. 
While surface finite element methods (SFEM), originating with Dziuk's finite element method 
for the Laplace--Beltrami operator and surveyed by Dziuk and Elliott \cite{Dziuk1988,DziukElliott2013}, 
are popular for general surface PDEs using low-order elements on triangulated surfaces, 
our approach leverages surface parameterization and high-order $hp$-FEM. 
This choice is motivated by the need for the extreme precision required in conformal mapping, 
where the conformal modulus and mapping must be computed to many significant digits, 
often in the presence of singularities where $hp$-adaptivity provides a significant 
advantage over standard low-order methods.

\subsection{Illustrative Example: 3D Face Recognition}\label{sec:appetizer}

Face recognition has many applications, especially in biometrics. Since a standard 2D
image recognition is fragile due to many factors, 3D methods have been proposed as a more reliable
alternative. Each 3D surface with suitable topology can be mapped conformally to a 2D domain, and thus, the image recognition becomes a conformal map comparison problem. See
\cite{szeptycki2010conformal}.

In Figure~\ref{fig:appetizer} one example of how such a process might work is illustrated. 
First, four points on the outer boundary are fixed. In our formulation, the canonical domain is
a quadrilateral with horizontal slits if holes are present. 
Second, the Laplace-Beltrami equation is solved twice with different boundary conditions. 
As mentioned above,
this process is only marginally slower if multiple holes are present. 
In the example, both eyes and the mouth are modeled as holes.

\begin{figure}
  \centering
  \subfloat{
  \includegraphics[height=2.5in]{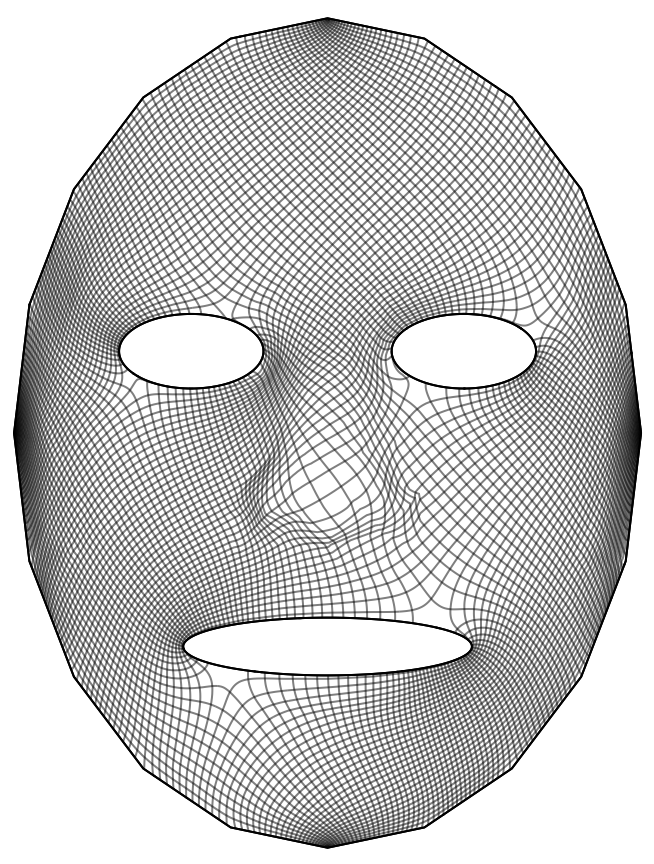}
  }
  \subfloat{
  \includegraphics[height=2.5in]{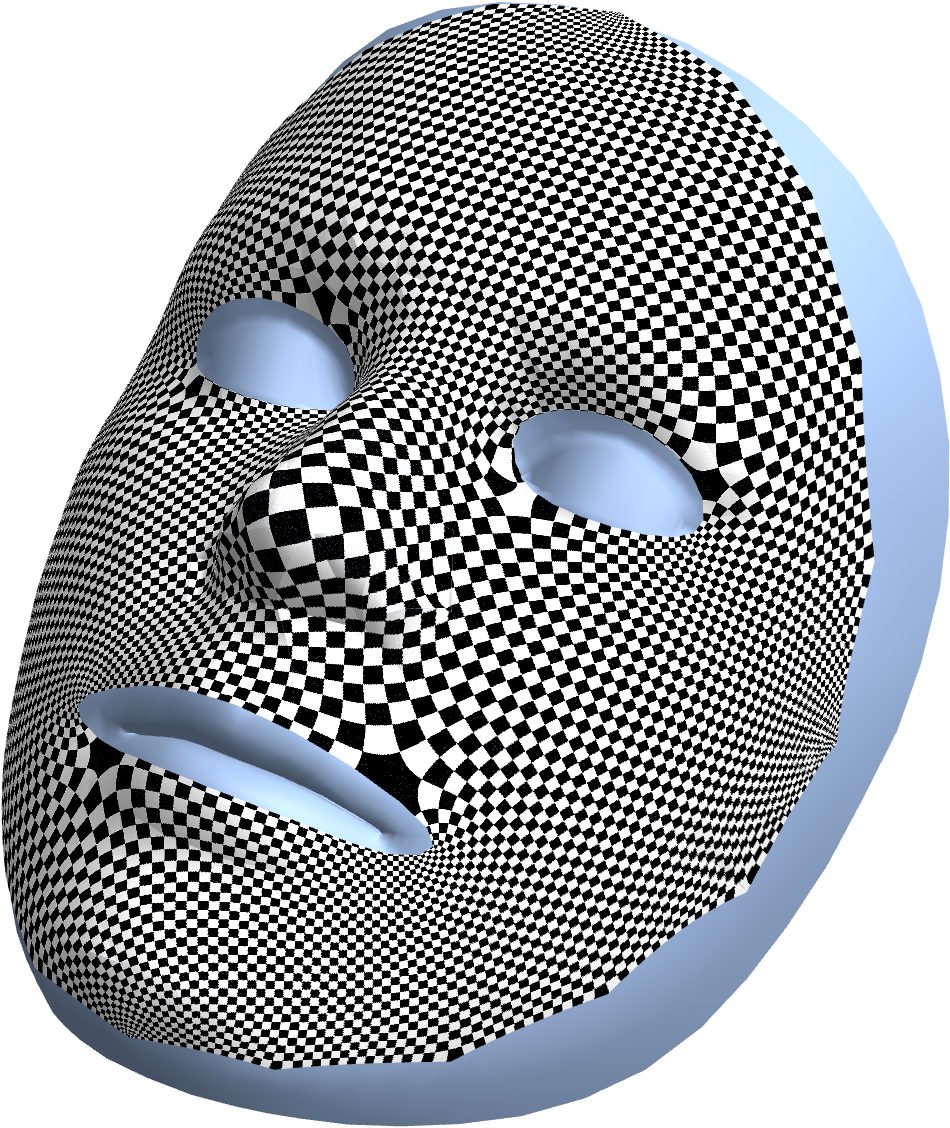}
  }
  \caption{Map of a face. Left: The map on the parameter space. Right: Checkerboard illustration of the
map on the surface. Original data source: \cite{snap2025facemesh}.} \label{fig:appetizer}
\end{figure}

\subsection{Organization}
The rest of the paper is organized as follows: 
After the introduction, preliminaries including the finite element method used are presented.
In Section~\ref{sec:cfms} the conjugate function method is introduced.
Section~\ref{sec:conjugateConstruction} discusses both the proof of the reciprocal
identity for oriented surfaces, and the efficient construction of the conjugate
problem.
Torii are given as a special class of closed surfaces in Section~\ref{sec:torii}.
Sections~\ref{sec:numex}-\ref{sec:polyhedral} cover series of numerical experiments
and applications, before conclusions at the end. 

\section{Preliminaries}
In this section, some basic concepts and definitions used throughout this paper are introduced. 

\subsection{Riemann surfaces and the Laplace-Beltrami Operator}
Let $X$ be a Hausdorff space that is locally homeomorphic to the Euclidean plane $\mathbb{R}^2$, which we also consider as a one-dimensional complex manifold. We assume that  its boundary $\Gamma$ is either empty or consists of a finite number of rectifiable arcs that may intersect only at their endpoints, and its closure $X\cup \Gamma$ is compact. A complex {\it chart} on $X$ is a homeomorphism $\varphi: U \to V$ from an open subset $U \subset X$ onto an open subset $V \subset \mathbb{C}$. Two complex charts $\varphi_j: U_j \to V_j$ (for $j = 1, 2$) are said to be holomorphically compatible if $U_1 \cap U_2=\emptyset$ or the mapping
$$
\varphi_2 \circ \varphi_1^{-1} : \varphi_1(U_1 \cap U_2) \to \varphi_2(U_1 \cap U_2)
$$
is conformal (biholomorphic). The collection of compatible charts covering the whole space $X$ is called an {\it atlas} of the manifold. We call manifolds with an atlas of holomorphically compatible charts a {\it Riemann surface}. 

Our task is to define the Laplacian on some surface $S$, that is, the Laplace-Beltrami operator $\Delta_S$, in a form suitable for finite element implementation.
In our setting, the surface is always assumed to be
given in some parameterized form. Let $\mathbf{x}_S : \Gamma \to S$ be a parameterization
of a surface $S$. The goal is to treat $\Gamma \subset \mathbb{R}^2$ as the reference domain on
which the finite elements are defined. Let $J_\mathbf{x}$ be the Jacobian of the mapping, and
hence $G_S = J_\mathbf{x}^T J_\mathbf{x}$ is the first fundamental form.

The tangential gradient of some function $v : S\to \mathbb{R}$ is
\begin{equation}
  (\nabla_S v)\circ \mathbf{x}_S := J_\mathbf{x} G_S^{-1} \nabla(v \circ \mathbf{x}_S),
\end{equation}
and $\Delta_S := \nabla_S \cdot \nabla_S$. The variational formulation on an image $K$ of a given element $T$ in a discretization of $\Gamma$ becomes
\begin{align}\label{eq:femvar}
\int_K \nabla_{K}\psi\cdot\nabla_{K} v\,dx &= \int_{T} \nabla(\psi \circ \mathbf{x}_K)^T  G_K^{-T} G_S G_K^{-1} \nabla(v \circ \mathbf{x}_K) \sqrt{\operatorname{det}(G_K)}d\tau.
\end{align}

\subsection{High-Order Finite Element Method}\label{sec:hpfem}
High-order finite element methods ($hp$-FEM) provide exponential convergence for conformal capacity problems if the discretization correctly reflects boundary singularities. We denote the weak solution of the Dirichlet-Neumann problem by $u_0$.
The following result by Babuška and Guo \cite{BaGuo1,BaGuo2} provides the theoretical foundation for our choice of method.

\begin{theorem} 
Let $\Omega \subset \mathbb{R}^2$ be a polygon and let $u_0$ be the weak solution in $H^1(\Omega)$. If $u_0$ belongs to a suitable countably normed space $B^2_\beta(\Omega)$ (where derivatives of high order are weighted by distance to singularities), then the $hp$-FEM solution $v$ satisfies:
\[
\inf_v \|u_0 - v\|_{H^1(\Omega)} \leq C\,\exp(-b \sqrt[3]{N}),
\]
where $N$ is the number of degrees of freedom. This bound holds for proper geometric meshes, where elements are refined geometrically towards singularities and polynomial degrees are increased linearly.
\end{theorem}

\subsubsection{Error Estimation}
The a posteriori error estimation method
specific to the solution method  applied here
is the so-called auxiliary subspace error estimation.
The estimate measures the projection of the residual to the auxiliary space,
that is, to a set of finite element degrees of freedom that extend
the approximation space in a natural way. One example of such extension
is to consider all edge modes of degree $p+1$ and bubble modes of 
$p+1$ and $p+2$, where $p$ is the (constant) polynomial degree used in the discretization.
For details, see \cite{hno}.

First, $p$-robustness has never been rigorously shown for this class of
error estimators. Yet, there exists compelling
numerical evidence that the method is indeed $p$-robust.
Second, by construction, these estimators are optimistic.
If properly constructed, they cannot overestimate the error.

\subsubsection{Special Types of Boundaries}\label{sec:refine}
The computational domains can be defined in different ways, either exactly in some parameterized
form or using some given discretization.
Our approach for geometric grading of the meshes 
with known locations of singularities is based on rule based algorithms \cite{Hakula2013}. 

One non-standard boundary type in the context of finite element method 
considered here is the slit. We define a slit to be a set of element nodes and edges
that form a loop with zero area. That is, the data structures support having
two edges on top of each other, both with their own kinematic constraints.  

\section{The Conjugate Function Method On Planar Domains and Riemann Surfaces}
\label{sec:cfms}

Let $\Omega\subset \Sf$ be a simply connected domain on a Riemann surface $\Sf$ so that the boundary of $\Omega$ is a Jordan curve.
We call $\Omega$ together with four positively oriented points $z_1,z_2,z_3,z_4\in \partial \Omega$ a (generalized) {\it quadrilateral} and denote it by $Q=(\Omega; z_1,z_2,z_3,z_4)$. The boundary segments connecting the pairs of points $(z_j,z_{j+1})$ for $j=1,2,3$, and $(z_4,z_1)$ for $j=4$, respectively, are denoted by $\gamma_j$. 

It is well-known (see e.g. \cite{Ahlfors}) that there exists a unique number $h>0$ called the {\it conformal modulus} of $Q$, such that there exists a conformal mapping of the rectangle $R_h=[0,1]\times[0,h]\subset \C$ onto $\Omega$, with boundary points $z_1,z_2,z_3,z_4$ corresponding to the images of the points $0,1,1+ih,ih$, respectively. The conformal modulus determines the conformal equivalence class of $\Omega$ in the sense that there exists a conformal mapping between quadrilaterals (with boundary point correspondence) if and only if they have the same modulus. In the following the conformal modulus of a quadrilateral $Q$ is denoted by $M(Q)$.

In this study, the conformal modulus of a quadrilateral is computed
via its connection to the Laplacian.
Recall that there exists a (unique) harmonic solution $u$ to the following Dirichlet-Neumann mixed boundary value problem:
\begin{equation}
\label{eq:dirneu}
\left\{
\begin{array}{rcl}
    \Delta_S u(z) = 0 & \text{for} & z \in \Omega, \\
    u(z) = 0 & \text{for} & z \in \gamma_2, \\
    u(z) = 1 & \text{for} & z \in \gamma_4, \\
    \partial u(z)/\partial n = 0 & \text{for} & z \in \gamma_1 \cup \gamma_3,
\end{array}
\right.
\end{equation}
where $n$ is the unit exterior boundary normal and
$\Delta_S$ is the Laplace-Beltrami operator. For $\Omega \subset \mathbb{C}$ the conformal modulus is connected to the above boundary value problem  by the identity (see e.g. Ahlfors \cite[Theorem~4.5]{Ahlfors} and Papamichael and Stylianopoulos \cite[Theorem~2.3.3]{PapamichaelStylianopoulos2010}):
\begin{equation}
\label{mod-dir}
M(Q)=\iint_{\Omega}|\nabla u|^2 \, dx \, dy.
\end{equation}

For a quadrilateral $Q=(\Omega; z_1,z_2,z_3,z_4)$ we call $\widetilde{Q}=(\Omega; z_2, z_3, z_4, z_1)$ its {\it conjugate quadrilateral} and the corresponding problem \eqref{eq:dirneu} for the quadrilateral $\widetilde{Q}$ the {\it conjugate Dirichlet--Neumann problem}. It is well-known that if $M(Q)=h>0$, then $M(\widetilde{Q})=1/h$, which leads to \begin{equation}
\label{eq:resiproc}
M(Q) M(\widetilde{Q}) = 1,
\end{equation}
for all quadrilaterals $Q$. This {\it reciprocal identity} is very useful since
it can be interpreted as an error estimate \cite{hrv}.
\begin{remark}[Reciprocal identity as an error estimate]
It is immediately clear that the identity is a necessary but not sufficient condition.
In the context of this paper it is important to notice that the same discretization is
used for both the primary and the conjugate problem. Since the main source of numerical
error arises from the Dirichlet-Neumann interaction in the corners, say, the underlying
discretisation is always in this sense symmetric. Within finite element methods
for the two solutions to have different convergence rates is highly unlikely.
The respective constants may vary, of course.
\end{remark}
Furthermore, we  may observe that the canonical conformal mapping of a quadrilateral $Q = (\Omega; z_1, z_2, z_3, z_4)$ onto the rectangle $R_h$ with vertices at $1+i h$, $i h$, $0$, and $1$, can be obtained by solving the corresponding Dirichlet--Neumann problem and its conjugate problem.

\begin{lemma}
Let $Q$ be a quadrilateral with modulus $h$, and suppose $u$ solves the Dirichlet--Neumann problem \eqref{eq:dirneu}. If $v$ is a harmonic function conjugate to $u$, satisfying $v(\operatorname{Re} z_3, \operatorname{Im} z_3)=0$, and $\tilde{u}$ represents the harmonic function solving the Dirichlet--Neumann problem for the conjugate quadrilateral $\widetilde{Q}$, then $v = h \tilde{u}$.
\end{lemma}

\subsection{Q-type canonical domains}
Let $h>0$, and let $ \big((\zeta_1, d_1),\ldots,(\zeta_N, d_N)\big)$ be complex numbers such that $\operatorname{Re} \zeta_i\in(0,1)$ and $\operatorname{Im} \zeta_i\in(0,h)$ for all $i=1,\ldots, N$, where $N\ge 1$. For the case of multiply connected planar domains, we will consider {\it canonical domains} of the type
\begin{equation}
\label{can-multidom}
\big((0,1)\times (0,h)\big)
\setminus
\bigcup_{j=1}^N
[\zeta_j,\zeta_j+d_j] \subset \mathbb{C},
\end{equation}
i.e., a rectangle $(1, 0) \times (0, h)$ with $N$ horizontal slits removed. Canonical domains of this type are called {\it quadrilateral-like} (or {\it Q-type}) in \cite{hqr2}.

\section{Fast Construction of the Conjugate Problem}\label{sec:conjugateConstruction}
The starting point is the quadrilateral $Q$. In the multiply connected case
there may be a finite number of holes each with its own boundary $\partial E_i$.
Initially, each interior boundary will have a zero Neumann boundary condition: 
$\partial u/\partial n=0$ on every $\partial E_i$. The task is to construct the
conjugate problem where each of the $\partial E_i$ is set to some 
constant potential $v_i$, in other words, a corresponding Dirichlet boundary condition is defined,
$v = v_i, \text{ on } \partial E_i,\ \forall i$.

The overall solution process can be compressed to three steps: 1. Solve the primary problem, 2. Construct the conjugate problem, 3. Solve the conjugate problem. That is, considering the primary problem
\begin{equation}
\label{multi-bvals}
  \begin{cases}
    \Delta_S u = 0,\quad \text{ in } \Omega, \\
    u = 0, \quad \text{ on } \gamma_1,  \\
    u = 1, \quad \text{ on } \gamma_3, \\
    \partial u/\partial n = 0, \quad  \text{ on } \gamma_2, \gamma_4, \\
    \partial u/\partial n = 0, \quad  \text{ on } \partial E_i,\ \forall i, \\
  \end{cases} \text{ leading to }
    \begin{cases}
    \Delta_S v = 0,\quad \text{ in } \Omega, \\
    v = 0, \quad \text{ on } \gamma_2, \\ 
    v = 1, \quad \text{ on } \gamma_4, \\ 
    v = v_i, \quad \text{ on } \partial E_i,\ \forall i, \\
    \partial v/\partial n = 0, \quad  \text{ on } \gamma_1, \gamma_3. 
  \end{cases}
\end{equation}

The values $v_i$ are not prescribed by the geometry alone. They are chosen so that the
Dirichlet energy of the conjugate problem is minimized subject to the fixed values
on $\gamma_2$ and $\gamma_4$. Equivalently, after the primary modulus has been
computed, the reciprocal identity discussed in the next subsection identifies the
correct conjugate modulus as $1/M(Q_*)$.
In our previous work, this task has been completed using optimization, which has not only 
been computationally expensive but has also led to unavoidable loss of accuracy 
when quadratic objective functions have been used. 
Remarkably, using the matrix representation of the discretization of the problem 
and the energy minimization criterion, one
can derive a two-step method for finding the minimizing Dirichlet boundary conditions
for the conjugate problem.

\subsection{Reciprocal Identity for Multiply Connected Domains}
We note that when the boundary conditions given by \eqref{multi-bvals} are applied to a Q-type canonical domain given by \eqref{can-multidom}, then we obtain the following natural generalization of the modulus of quadrilateral.

Let $\Omega$ be a domain of the type \eqref{can-multidom} on an oriented surface $S$. Denote by $Q_*=(\Omega;z_1,z_2,z_3,z_4)$ the Q-type domain with points $z_1,z_2,z_3,z_4$ chosen in a positive order from the same boundary component. Then the conformal modulus of $Q_*$ and its modulus and $\tilde Q_*=(\Omega;z_2,z_3,z_4,z_1)$, the conjugate modulus, are defined by 
\begin{equation}
M(Q_*)=\iint_{\Omega}|\nabla u|^2 \, dx \, dy,
\quad
M(\tilde Q_*)=\iint_{\Omega}|\nabla v|^2 \, dx \, dy,
\end{equation}
where $u,v$, respectively, are harmonic solutions to the boundary value problems given by \eqref{multi-bvals}.

These definitions are natural generalizations of the conformal modulus of a quadrilateral for multiply connected domains, and the classical definition discussed in Section \ref{sec:cfms} can be understood as their special cases. Furthermore, they admit the following very useful generalization of the reciprocal identity \eqref{eq:resiproc}:

\begin{proposition}
Let $\Omega$ be a domain on an oriented surface $S$ that can be mapped conformally onto a Q-type canonical domain $Q_*=(\Omega; z_1,z_2,z_3,z_4)$ defined by \eqref{can-multidom}, where the points $z_1,z_2,z_3,z_4$ are chosen from the same boundary component of $\Omega$. Then the generalized conformal moduli of $\Omega$ satisfy the reciprocal identity:
\begin{equation}
M(Q_*)M(\tilde Q_*)=1.
\end{equation}
\end{proposition}

\begin{proof}
First recall that the Dirichlet energy integral
\begin{equation}
\iint_{\Omega}|\nabla u|^2 \, dx \, dy
\end{equation}
in \eqref{mod-dir} is a well-known conformal invariant in the plane, and this property extends also to conformal mappings between Riemann surfaces (see, e.g., \cite[Section 3.6]{JostRiemann}). Because the boundary data in \eqref{multi-bvals} is preserved under conformal mappings, it is sufficient to show that this identity is true for the canonical domains given by \eqref{can-multidom}.

But in this case, the solution of the boundary value problem \eqref{multi-bvals} is given by the harmonic functions $u(x,y)=x$ and $v(x,y)=y/h$. Therefore, the desired identity follows immediately.
\end{proof}

\begin{remark}
Previously, the use of reciprocal identity as a  method of error estimation has been verified only in the simply connected case \cite{hqr}.
\end{remark}

\subsection{Linear Algebra Based Construction of the Conjugate Problem}
As indicated in \eqref{multi-bvals}, the discretized primary and conjugate
problems share the same degrees of freedom before application of boundary conditions.
The degrees of freedom admit a partition that can be used to define
a block structure for the discretized matrix $A$.
For the conjugate problem the degrees of freedom on 
the four boundary segments $\gamma_i$ 
defining the quadrilateral $Q$ are denoted by $D_0, D_1, N^0$, and $N^1$,
where $D_0$ indicates $v=0$ and $N^0$ zero Neumann on the segment with $u=0$
on the primary problem.
The internal degrees of freedom are denoted by $B$.
Finally, the degrees of freedom on the boundaries of the holes are 
denoted by $E_i$ .
\begin{remark}
  Since the potential is constant over the boundary of every hole in the conjugate problem,
  all edge degrees of freedom are set to zero. This is a technical detail that does not affect the discussion below.
\end{remark}
Using the notation given above the discretized system has the following block structure
\[
A =
\begin{pmatrix}
A_{BB}     & A_{B N^1} & A_{B N^0} & A_{B D_1} & A_{B D_0}        & A_{B E_1} & \dots &   A_{B E_n} \\
A_{N^1 B} & A_{N^1 N^1} & A_{N^1 N^0} & A_{N^1 D_1} & A_{N^1 D_0} & A_{N^1 E_1} & &   A_{N^1 E_n}\\
A_{N^0 B} & A_{N^0 N^1} & A_{N^0 N^0} & A_{N^0 D_1} & A_{N^0 D_0} & A_{N^0 E_1} & &   A_{N^0 E_n}\\
A_{D_1 B} & A_{D_1 N^1} & A_{D_1 N^0} & A_{D_1 D_1} & A_{D_1 D_0} & A_{D_1 E_1} & &   A_{D_1 E_n}\\
A_{D_0 B} & A_{D_0 N^1}  & A_{D_0 N^0} & A_{D_0 D_1} & A_{D_0 D_0}& A_{D_0 E_1} & &   A_{D_0 E_n}\\
A_{E_1 B} & A_{E_1 N^1}  & A_{E_1 N^0} & A_{E_1 D_1} & A_{E_1 D_0}& A_{E_1 E_1} & &   A_{E_1 E_n}\\
\vdots &   &  &  &  &  & \ddots \\
A_{E_n B} & A_{E_n N^1}  & A_{E_n N^0} & A_{E_n D_1} & A_{E_n D_0}& A_{E_n E_1} & \dots &   A_{E_n E_n} \\
\end{pmatrix}.
\]
The solution vector $x$ (or the potential $v$) has a similar structure
\[
x = \begin{pmatrix}
x_{B} & x_{N^1}  & x_{N^0} & x_{D_1} & x_{D_0}& x_{E_1} & \dots &   x_{E_n}  
\end{pmatrix}^T.
\]
In order to derive the quadratic form we further divide the degrees of freedom into two classes: free ($B$, $N^0$, $N^1$) and fixed (all Dirichlet).
These classes are denoted by $I$ and $D$, respectively. 
Then the matrix $A$ has a simple $2\times 2$ form
$
A=\left(\begin{smallmatrix}
A_{II} & A_{ID} \\
A_{DI} & A_{DD} \\
\end{smallmatrix}\right)
$, and $x = \left(\begin{smallmatrix}
x_{I} & x_{D}  
\end{smallmatrix}\right)^T$. Using the block form one can solve (formally) the free degrees of freedom
\[
A_{II}x_I = -A_{ID} x_D, \quad
x_I = -A_{II}^{-1} A_{ID} x_D.
\]
Therefore, one can simplify the quadratic form
\begin{equation}\label{eq:conjquadratic}
Q_F = \begin{pmatrix}
x_{I}\\
x_{D} \\
\end{pmatrix}^T
\begin{pmatrix}
A_{II} & A_{ID} \\
A_{DI} & A_{DD} \\
\end{pmatrix}\begin{pmatrix}
x_{I}\\
x_{D} \\
\end{pmatrix} =
x_D^T A_{DD}x_D - x_D^T A_{DI} A_{II}^{-1} A_{ID} x_D.
\end{equation}
Let us next denote the vector of unknown parameter values 
$
x'_D = \left(\begin{smallmatrix}
v_1 & \dots &   v_n  
\end{smallmatrix}\right)^T.
$
Using the fact that every nonzero degree of freedom on $E_i$ is constant
and equal to $v_i$,
we arrive at the quadratic minimization problem
\begin{equation}
  \text{min}_{x'_D \in \mathbb{R}^n} \{{x'_D}^T K x'_D- b^T x'_D\},
\end{equation}
where $K$ and $b$ are the coefficient matrix (symmetric) and vector, respectively, 
extracted from \eqref{eq:conjquadratic}. We conclude that the problem of
finding the Dirichlet boundary values of the holes reduces to finding
a solution to an $n\times n$ linear system of equations.
Of course, this construction requires that the necessary bookkeeping of degrees of freedom
is available.

\subsection{Details of the Reduction Step}

The Dirichlet vector $x_D$ has a natural partition where the zero Dirichlet part $x_{D_0}$ is omitted,
with lengths $k_0,k_1,\ldots,k_n$ such that 
$m~=~\sum_{j_0}^n k_j$:
\[
x_D = \begin{pmatrix}
x_{D_1} & x_{E_1} & \dots &   x_{E_n}  
\end{pmatrix}^T.
\]

The idea is to construct a reduction matrix $R$ which maps, in other words, sums,
all coefficients of the dofs $x_{E_j}$ to those of $v_j$.
The matrix $R \in \mathbb{R}^{m\times (n+1)}$ has the form
\[
R = \begin{pmatrix}
\mathbf{1}_{k_0} & 0_{k_0} & \cdots &  0_{k_0} \\
0_{k_1} & \mathbf{1}_{k_1} &  &  0_{k_1} \\
\vdots & & \ddots & \vdots \\
0_{k_{n}}  & 0_{k_{n}} &\cdots &\mathbf{1}_{k_n}
\end{pmatrix}.
\]

The matrix $A_{DD}$ is block diagonal, thus $K_0 = R^T A_{DD} R$ is a diagonal matrix.
The second part $K_1 = R^T A_{DI} A_{II}^{-1} A_{ID} R$ is a symmetric matrix.
Then, using the colon notation, we get
\[
b = K_1(1,2:)^T, \qquad K = K_1(2:,2:)-K_0(2:,2:).
\]
The constant term does not affect the minimization.

\subsection{Effect of the Improved Construction: Two Holes on a Sphere Revisited}
In our earlier work \cite{24M1656840} we considered a problem with two holes on a sphere.
There the conjugate problem was constructed using our original method that was constrained
by the optimization process. 
Here we first apply the quadratic minimization and, indeed, obtain more accurate results 
using exactly the same finite element discretization. 
This is illustrated in Figure~\ref{fig:twoholes}.
What is even more remarkable is that this direct construction has a well-defined
a priori computational complexity that is often significantly smaller than
that of standard optimization. 

The precise description of the problem on the parameter plane is as indicated in 
Table~\ref{tbl:twoholoes}. 
The moduli obtained with the new construction of the conjugate problem are
\begin{equation}
  M(Q) = 0.7901908, \; M(\widetilde{Q}) = 1.2655173,
\end{equation}
with $\operatorname{reci}(Q) = 7 \times 10^{-8}$.

\begin{table}
  \caption{Two Holes: Unit circle centered at $(1/2,1/2)$. The locations, radii, and the obtained potential
values in the construction of the conjugate problem are given for both holes. }\label{tbl:twoholoes}
  \centering
  \begin{tabular}{llll}
    \toprule
    Hole & Center & Radius & $v_k$ on $\partial B_k$  \\ \midrule
    $B_1$ & $(1/4,1/4)$ & $1/4$ & $0.5344459257688663$ \\ 
    $B_2$ & $(3/4,3/4)$ & $1/4$ & $0.4883797948571419$ \\ \bottomrule
  \end{tabular}
\end{table}

\begin{figure}
  \centering
  \includegraphics[width=0.55\textwidth]{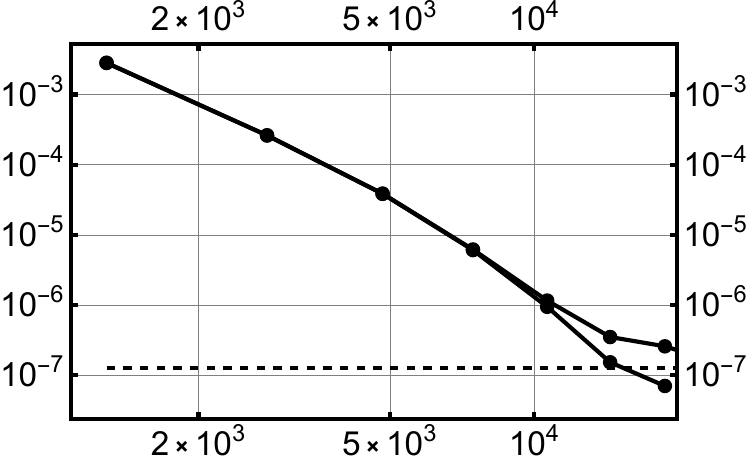}
  \caption{Two Holes Revisited. Convergence of the reciprocal error as a function of degrees of freedom
  using optimization and the new construction. 
  The current version is eventually more accurate on exactly the same mesh, and dips below
  the threshold limit of $10^{-7}$ that indicates the best possible accuracy when the black-box optimization routine was used (loglog-plot) \cite{24M1656840}.}
  \label{fig:twoholes}
\end{figure}

\section{Torii} \label{sec:torii}
The  conjugate function method is not directly applicable
on closed surfaces of genus 1 or higher.
However, if the parameter domain can be decomposed
into admissible parts, then the conformal map can be constructed by considering each
part independently. These surfaces appear frequently in soft matter physics \cite{Noguchi2015ToroidalVesicles}.
\subsection{Torus: Genus 1}
For a regular torus the parameter domain is simply $[0,2\pi]\times [0,2\pi]$.
Due to symmetry, one admissible partition is to divide the
parameter domain into four parts. The resulting map and the convergence graph
are shown in Figure~\ref{fig:torus}.
The parameterisation of the torus, with inner radius $r = 1/2$, and outer radius 
$R = 3/2$ is
\begin{equation}
  T(u,v)=\left(
  \left(\frac{\cos (u)}{2}+1\right) \cos (v),\left(\frac{\cos (u)}{2}+1\right) \sin
   (v),\frac{\sin (u)}{2}
   \right).
\end{equation}
It is known that this parameterisation is not isothermal. 
This is immediately visible in
Figure~\ref{fig:torusC}. For $T(u,v)$, one has
\[
J_\mathbf{x} =\left(
\begin{array}{cc}
 -\frac{1}{2} \sin (u) \cos (v) & -\left(\left(\frac{\cos (u)}{2}+1\right) \sin
   (v)\right) \\
 -\frac{1}{2} \sin (u) \sin (v) & \left(\frac{\cos (u)}{2}+1\right) \cos (v) \\
 \frac{\cos (u)}{2} & 0 \\
\end{array}
\right),
\]
\[
G_K^{-1}=\left(
\begin{array}{cc}
 4 & 0 \\
 0 & \frac{4}{\cos ^2(u)+4 \cos (u)+4} \\
\end{array}
\right),
\]
\[
\sqrt{\operatorname{det}(G_K)} =\frac{1}{4} \left(\cos(u)+2\right).
\]
resulting in
\begin{equation}\label{eq:nonisothermal}
G_K^{-T} G_S G_K^{-1} \sqrt{\operatorname{det}(G_K)} = \left(
\begin{array}{cc}
 4 & 0 \\
 0 & \frac{4}{(\cos (u)+2)^2} \\
\end{array}
\right) \neq I.
\end{equation}
Indeed, with a suitable parameterization of $v$ one can find an isothermal
parameterisation for a torus.

\begin{figure}
  \centering
  \includegraphics[width=0.55\textwidth]{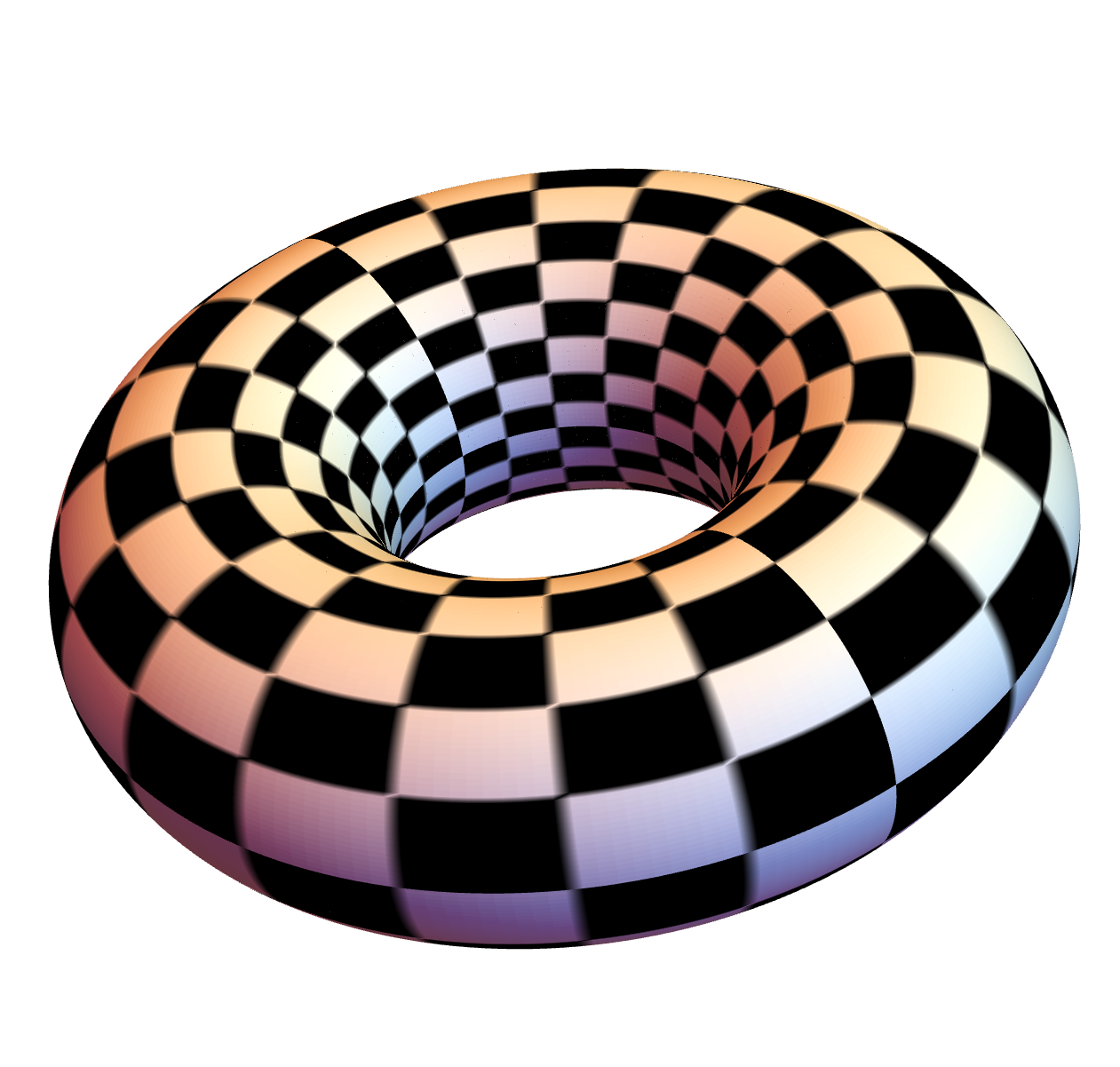}
  \caption{Torus with checkerboard colouring of the conformal map.}\label{fig:torusCB}
\end{figure}

\begin{figure}
  \centering
  \subfloat[{$(u,v) \in [0,\pi]\times [0,\pi]$}]{\label{fig:torusC}
  \includegraphics[height=1.5in]{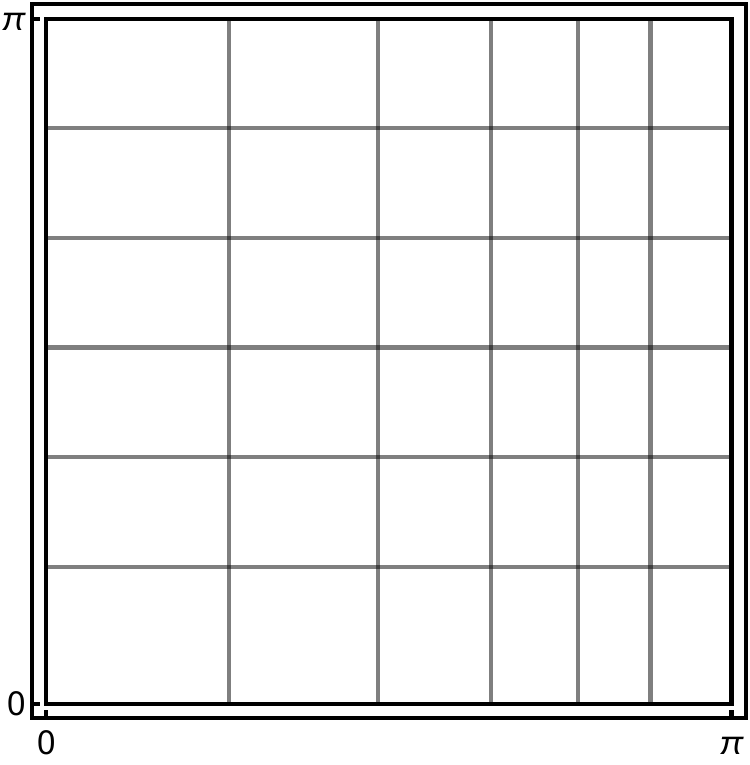}
  }
  \subfloat[Error estimates vs $N$.]{
  \includegraphics[height=1.5in]{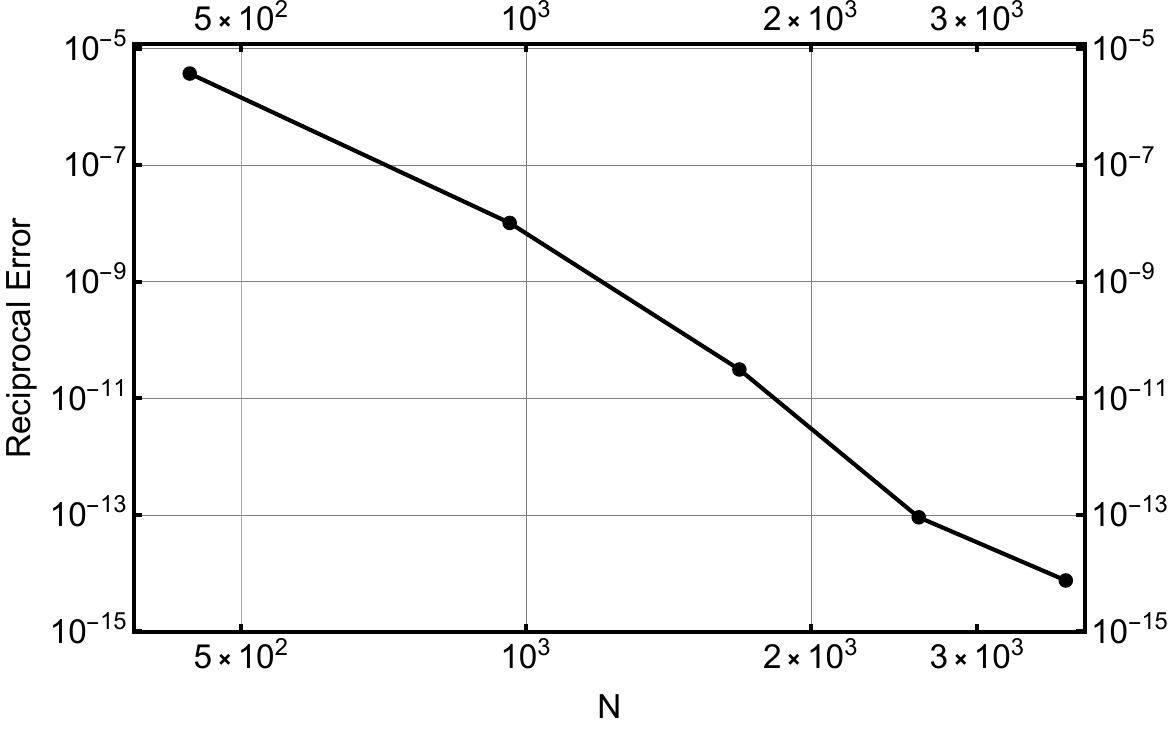}
  }
  \caption{Torus. (a) Map on the parameter space. (b) Convergence in the non isothermal parameterisation, error estimates vs. $N$  (the number of degrees of freedom). 
  Reciprocal error estimate: Solid line.
}\label{fig:torus}
\end{figure}

\subsection{Two examples: Genus 2 and Genus 3}
Finding conformal mappings on arbitrary surfaces is equivalent to
quad meshing on surfaces of arbitrary genus. Euler characteristic
cannot be satisfied except for a case of simple torus. In other words,
extraordinary points are inevitable in the general case and at best one can find
a quasiconformal map.
The conjugate function method can be applied to surfaces of higher genus
provided that suitable charts can be found. 
One pair of such realizations is shown in Figure~\ref{fig:DTMap}.
The 2-torus is defined by
\begin{equation}
  \frac{3}{32} \left(x^2+2 y^2\right)^2+\frac{1}{2
   \left(\left(x-\frac{3}{4}\right)^2+y^2\right)+\frac{1}{2}}+\frac{1}{2
   \left(\left(x+\frac{3}{4}\right)^2+y^2\right)+\frac{1}{2}}+2 z^2=\frac{13}{10}
\end{equation}
and the 3-torus by
\begin{align}
\frac{1}{24} \left(x^2+y^2\right)^2+\frac{2}{4 x^2+4 (y-2) y+5}+ & \frac{2}{4 x
 \left(x-\sqrt{3}\right)+4 y (y+1)+5}+ \\
 & \hspace{-5ex}\frac{2}{4 x \left(x+\sqrt{3}\right)+4 y \nonumber
 (y+1)+5}+z^2=\frac{13}{10}.
\end{align}

\begin{figure}
  \centering
  \subfloat{
    \includegraphics[height=2in]{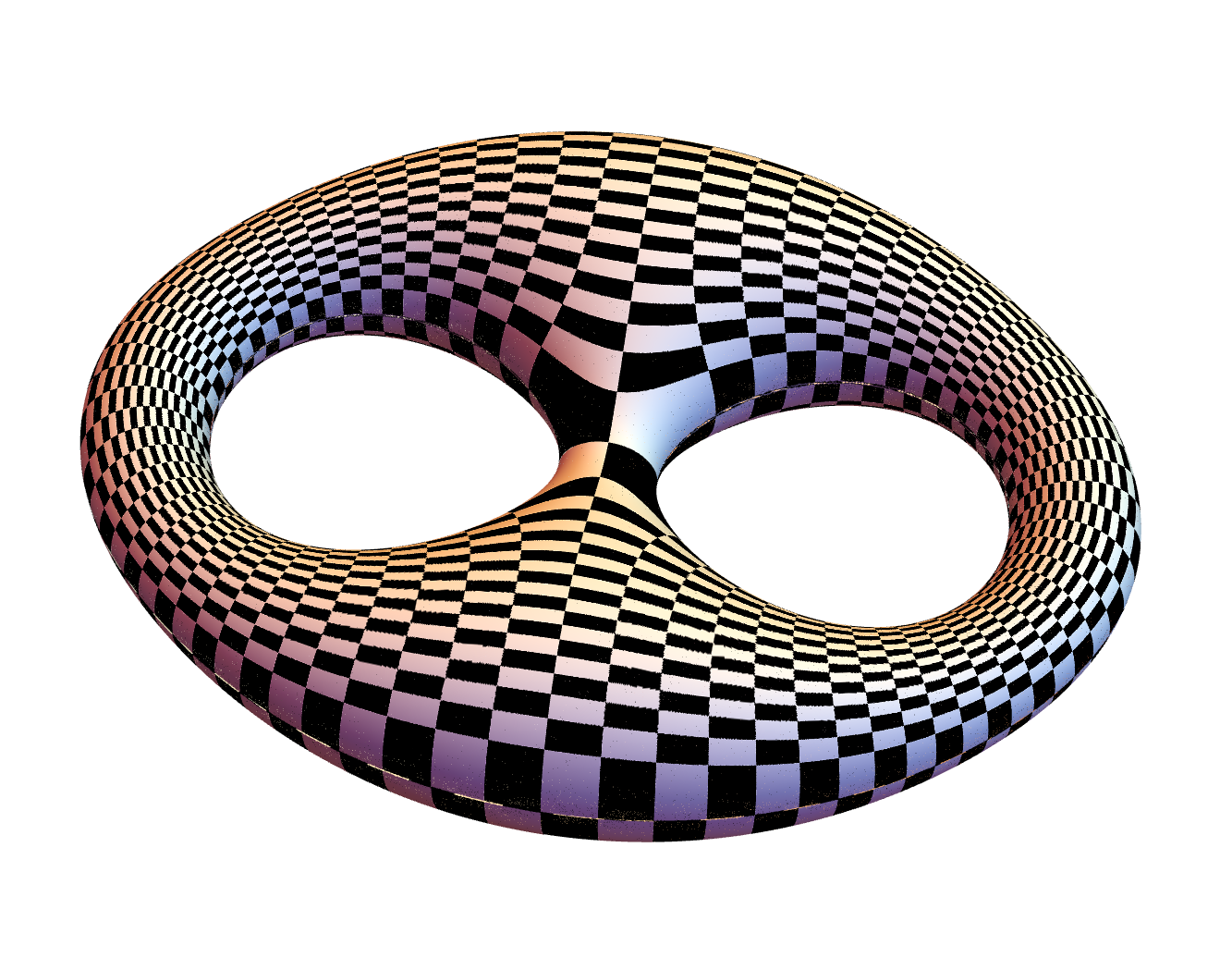}
    \includegraphics[height=2in]{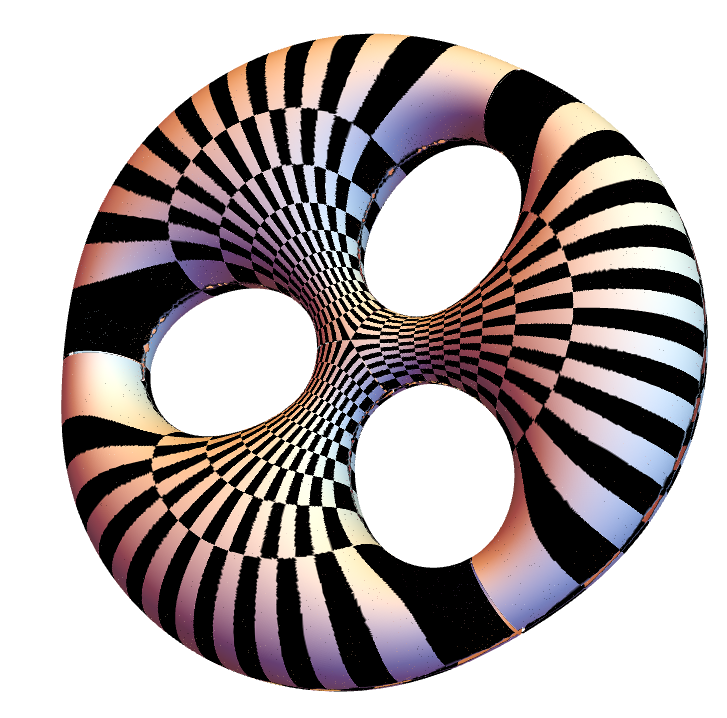}
  }
  \caption{Double and Triple Tori. Possible realizations.}\label{fig:DTMap}
\end{figure}

Let us consider the 2-torus in more detail. What is it that we see in Figure~\ref{fig:DTMap}?
The image is composed of eight sections using symmetries.
Let $k_f$ be the number of edges of a face $f$ and $\operatorname{val}(v)$ the valence of a vertex $v$.  
Since
\begin{equation}\label{eq:euler}
\sum_{v} \operatorname{val}(v) \;=\; 2E \;=\; \sum_{f} k_f,
\end{equation}
and Euler’s relation states that $\chi = V - E + F$, one obtains
\begin{equation}\label{eq:euler2}
\sum_{v}\bigl(4-\operatorname{val}(v)\bigr)\;-\;\sum_{f}\bigl(k_f-4\bigr)\;=\;4\chi.
\end{equation}
Equivalently,
\[
\sum_{v}\bigl(4-\operatorname{val}(v)\bigr)\;=\;4\chi \;+\;\sum_{f}\bigl(k_f-4\bigr).
\]
Using standard terminology, the left-hand side is the \emph{sum of vertex defects} that the extraordinary vertices contribute, and  
$\sum_{f}(k_f-4)$ is the \emph{sum of face defects}, and specifically a pentagon gives $+1$.
For a closed orientable surface of genus $g$, 
\[
\chi \;=\; 2 - 2g.
\]
With $g=2$, we have $\chi = -2$, so $4\chi = -8$.
If all faces are quads ($k_f=4$ for all $f$), then $\sum_f(k_f-4)=0$.  
Now, \eqref{eq:euler2} gives
\[
\sum_v (4-\operatorname{val}(v)) \;=\; -8,
\]
so we need for a valid subdivision of the surface, for example, eight pentagons. By close inspection of the
checkerboard image of the 2-torus, one can see that there is, indeed, one pentagon per every symmetric section, and
hence, we conclude that the construction is valid.

\section{Numerical Experiments on Domains with Multiple Boundary Components}
\label{sec:numex}
The efficiency of the conjugate function method on surfaces has been established already before \cite{24M1656840}.
In this section the focus is on problems with multiple boundary components both on planar domains and surfaces.
The first experiment was originally proposed by Nasser \cite{Nasser}. The latter two are variations of the same theme:
slit domains with strong singularities. First a selection of slits in a rectangle is considered and then lifted
onto a hemisphere that passes through the corner points. All geometric specifications are available from the authors upon request.
The three cases are outlined in Table~\ref{tab:experiments}.

\begin{table}
  \caption{Description of the three numerical experiments. The number of elements includes both triangles and quadrilaterals.
The polynomial order $p$ is the maximal used in the experiments and the number of degrees of freedom
corresponds to this value. In Cases B and C, the mesh used is exactly the same. }\label{tab:experiments}
  \centering
  \begin{tabular}{llrrrr}
    \toprule
    Case & Surface & \#holes & \#(nodes,edges,elements) & $p$ & \# dof \\ \midrule
    A & plane      & 35 & (23336,64018,40639) & 8  & 2,062,192\\ 
    B & plane      & 50 & (8931,17080,8100)   & 10 & 1,163,711\\ 
    C & hemisphere & 50 & (8931,17080,8100)   & 10 & 1,163,711\\ \bottomrule
  \end{tabular}
\end{table}
\subsection{Case A: Nasser's Challenge}
In this example 35 holes are punched inside a disk. The holes are carefully selected to include
either, one, two, or four reentrant corners or cusps.

Together with the map on the domain (Figure~\ref{fig:nasser}) and the canonical domain (Figure~\ref{fig:nasserDataA})
one can trace the paths from edge to edge where the gridlines do or do not touch any of the holes.
\begin{figure}
  \centering
  \includegraphics[width=0.6\textwidth]{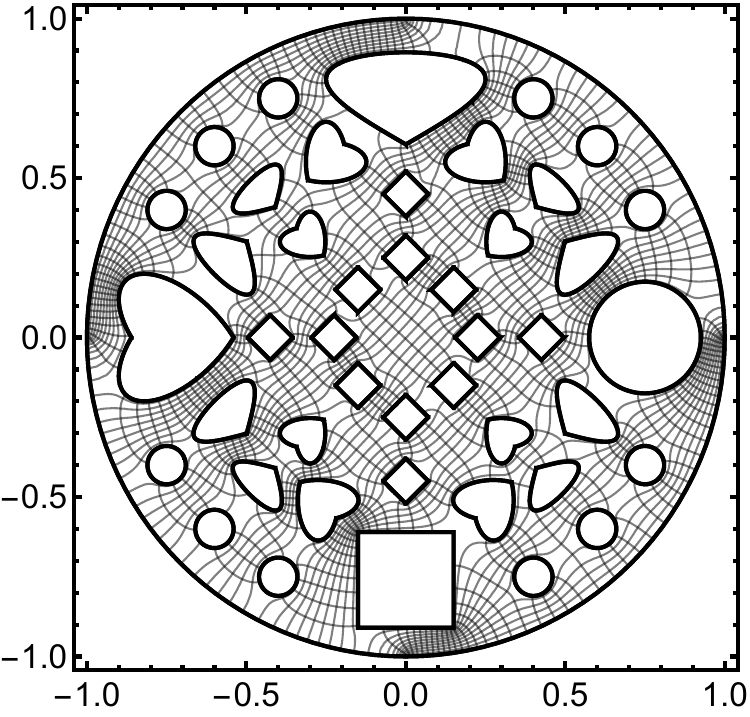}
  \caption{Nasser's Configuration: Map}\label{fig:nasser}
\end{figure}
The convergence graphs are shown in Figure~\ref{fig:nasserDataC}. 
Since the mesh (see Figure~\ref{fig:nasserDataB}) is generated on the original boundary discretisation,
we cannot expect exponential convergence. Indeed, even though the reciprocal error is relatively small,
the observed rate is only algebraic.
The auxiliary subspace estimates for the primary and conjugate problems
are optimistic. Notice that since the reciprocal error measures a product, 
it should stay above the auxiliary subspace estimates of both the primary and the
conjugate problems.

\begin{figure}
  \centering
  \subfloat[Canonical domain.]{\label{fig:nasserDataA}
    \includegraphics[width=0.45\textwidth]{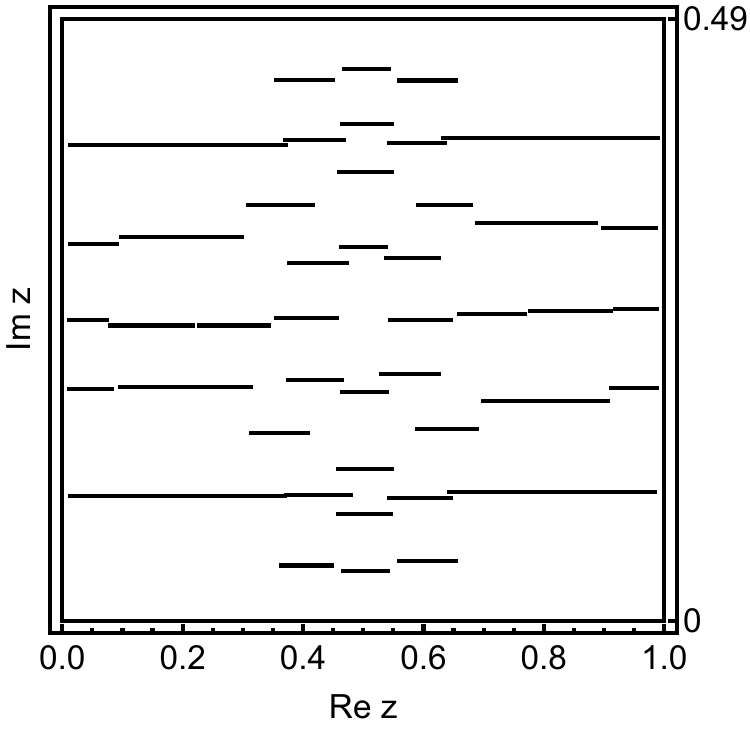}
  }
  \subfloat[Finite element mesh.]{\label{fig:nasserDataB}
    \includegraphics[width=0.45\textwidth]{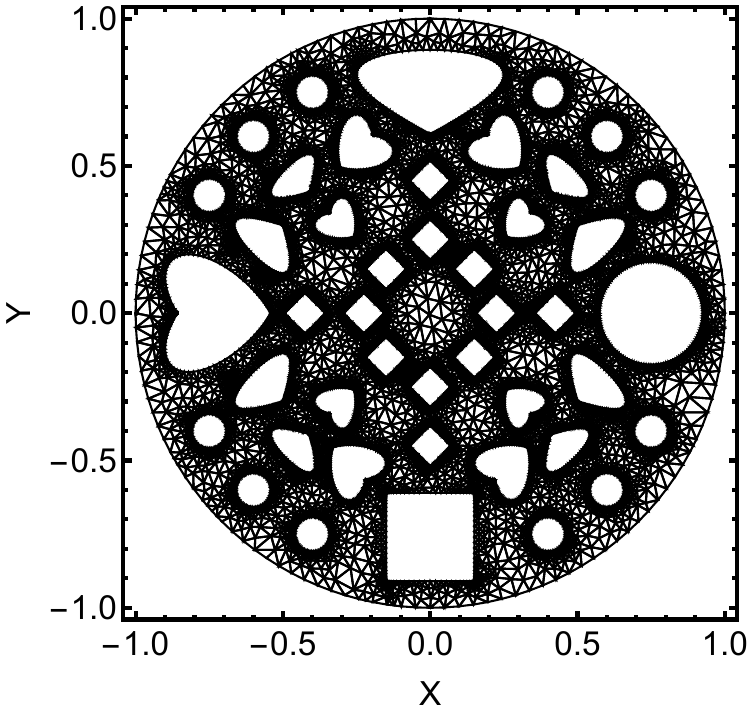}
  }
  \\
  \subfloat[{Convergence: Error estimates vs $N$.}]{\label{fig:nasserDataC}
    \includegraphics[height=1.75in]{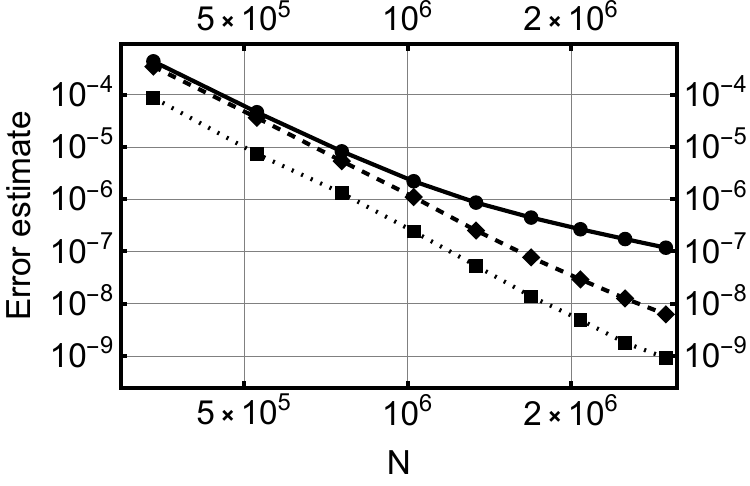}
  }
  \caption{Nasser's Configuration. (a) Canonical domain. (b) Mesh used in computations.
(c) Convergence graphs for error estimates. $N$ is the number of degrees of freedom that depends on the mesh and the polynomial order.
  The data points represent the results on a fixed mesh with the constant polynomial order ranging from $p=2,\ldots,10$.
Solid line: Reciprocal error; Dashed and dotted: Auxiliary subspace estimates for the primary and conjugate problems, respectively (loglog-plot).}\label{fig:nasserData}
\end{figure}

\subsection{Case B: Random Segments Inside a Rectangle}
As discussed above, with proper grading of the meshes, the $hp$-FEM should converge exponentially even
when strong singularities are present. Here 50 slits are placed at random in a $40\times 40$ rectangle
resulting in 100 $2\pi$-corners inside the domain.

The resulting map is shown in Figure~\ref{fig:segmentsA}. The detail shown in Figure~\ref{fig:segmentsAb}
shown how the slit perturbs the ``flow lines'' locally. The slit canonical domain is illustrated in 
Figure~\ref{fig:segmentsBa}. Together with the map on the domain (see Figure~\ref{fig:segmentsAa})
one can trace the paths from edge to edge where the gridlines do or do not touch any of the slits.

\begin{figure}
  \centering
  \subfloat[Conformal map on the domain.]{\label{fig:segmentsAa}
  \includegraphics[height=1.5in]{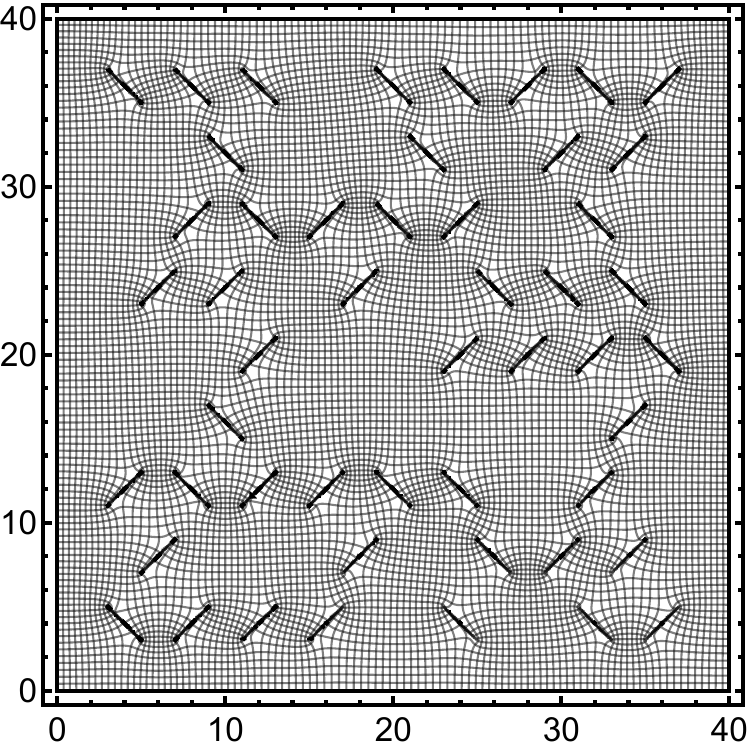}
  }\quad
  \subfloat[Detail of the map.]{\label{fig:segmentsAb}
  \includegraphics[height=1.5in]{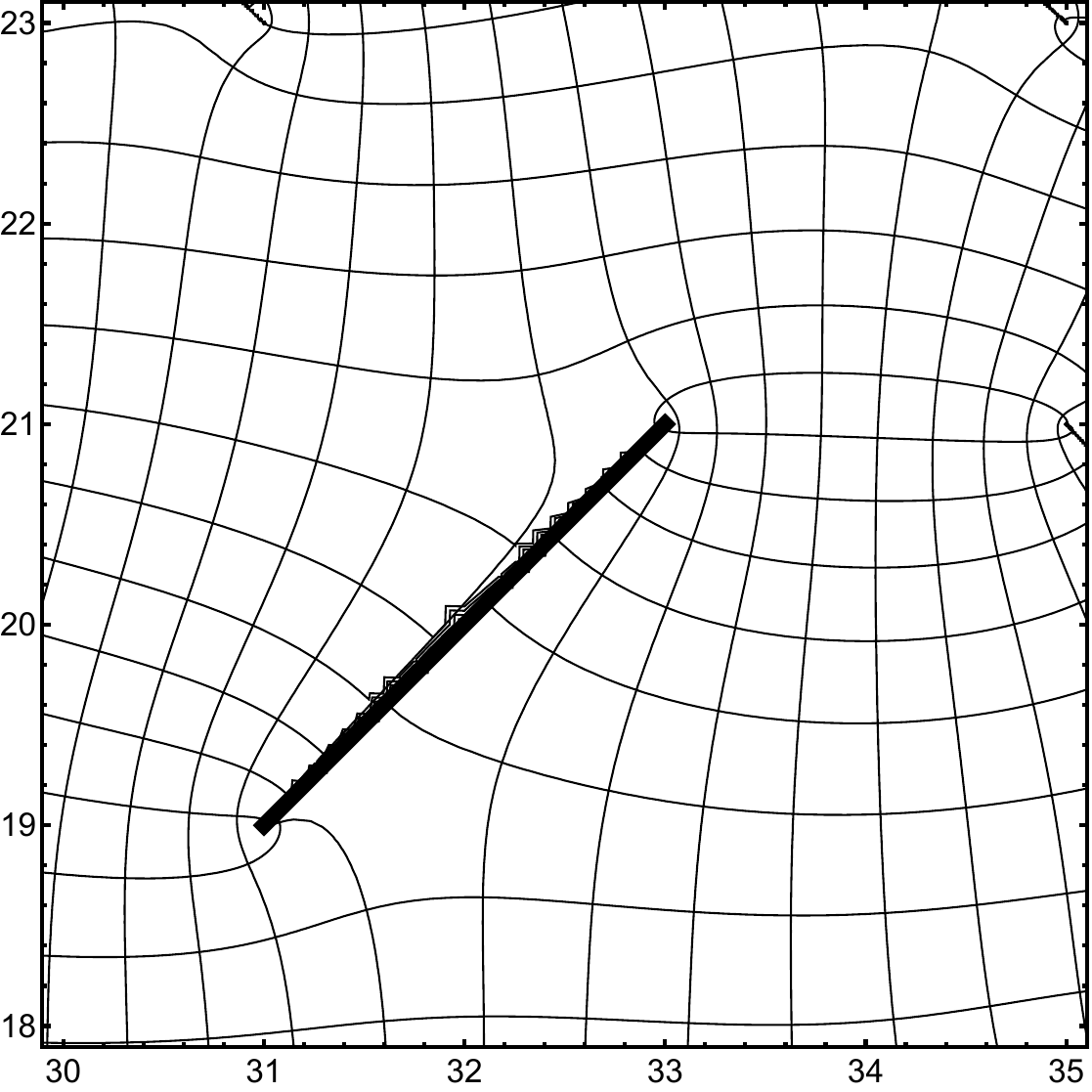}
  }
  \caption{Random Segments. Conformal map with a detail.}\label{fig:segmentsA}
\end{figure}

The convergence graphs are shown in Figure~\ref{fig:segmentsBb}. On a graded mesh with a constant polynomial order,
we observe exponential convergence in the reciprocal error. The auxiliary subspace estimates for the primary and conjugate problems
exhibit the same rate giving us high confidence in the results.

\begin{figure}
  \centering
  \subfloat[Canonical domain.]{\label{fig:segmentsBa}
  \includegraphics[height=1.5in]{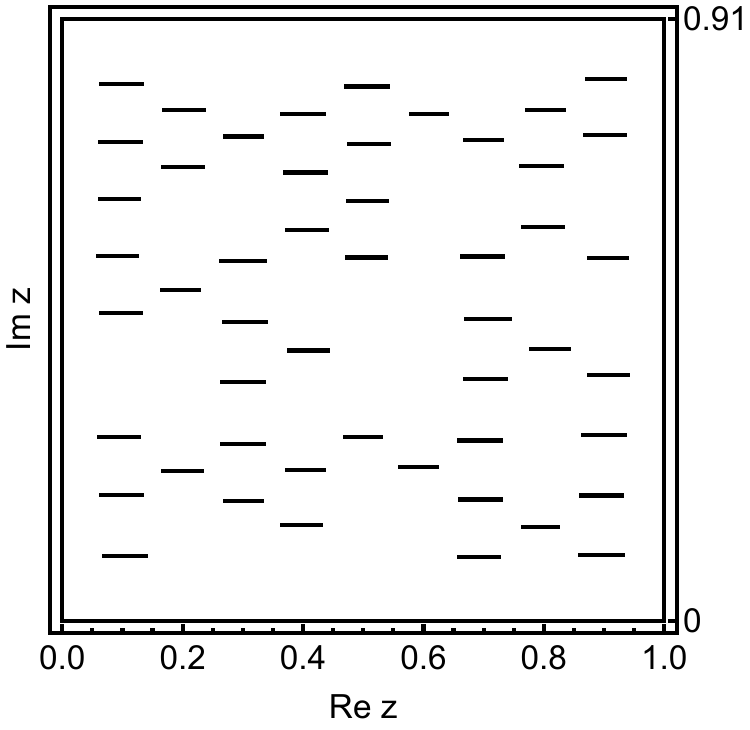}
  }
  \subfloat[{Convergence: Error estimates vs $N$.}]{\label{fig:segmentsBb}
  \includegraphics[height=1.5in]{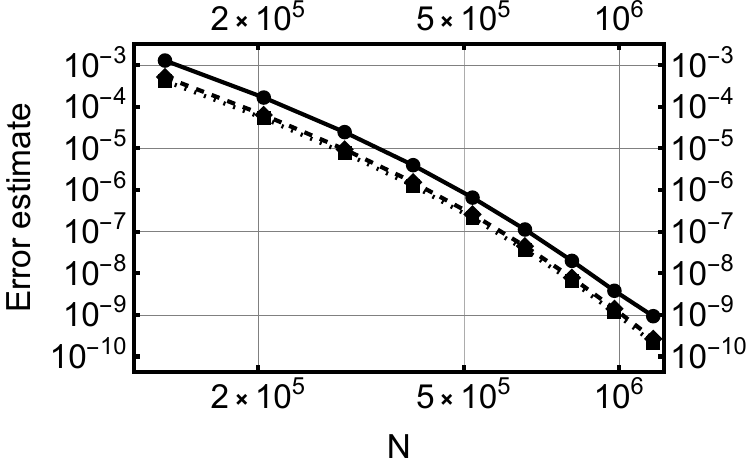}
  }
  \caption{Random Segments. (a) Canonical slit domain. (b) Convergence graphs for error estimates. $N$ is the number of degrees of freedom that depends on the mesh and the polynomial order.
  The data points represent the results on a fixed mesh with the constant polynomial order ranging from $p=2,\ldots,10$.
Solid line: Reciprocal error; Dashed and dotted: Auxiliary subspace estimates for the primary and conjugate problems, respectively (loglog-plot).}\label{fig:segmentsB}
\end{figure}

\subsection{Case C: Random Segments on a Hemisphere}
We continue with the same planar configuration. However, now the configuration represents the parameter space that is lifted onto a
conforming hemisphere (Figure~\ref{fig:ssegmentsA}). 
Interestingly, the convergence behaviour of the $hp$-FEM is practically unchanged despite the added
complexity of the Laplace-Beltrami problem (Figure~\ref{fig:ssegmentsBc}). 
This result underlines the fact that the method for constructing the conjugate problem
depends only on the discretisation and only indirectly on the underlying variational formulation. 

Of particular interest is the comparison of the two canonical domains shown in Figure~\ref{fig:ssegmentsBb}.
If one considers the two configurations, planar and surface, as extremal configurations, it is clear that the
intermediate stages in continuous lifting would result in smooth transition from one to another.
This immediately suggests applications based on transformations on the canonical domain.

\begin{figure}
  \centering
  \subfloat[Conformal map on the parameter domain.]{
  \includegraphics[height=1.5in]{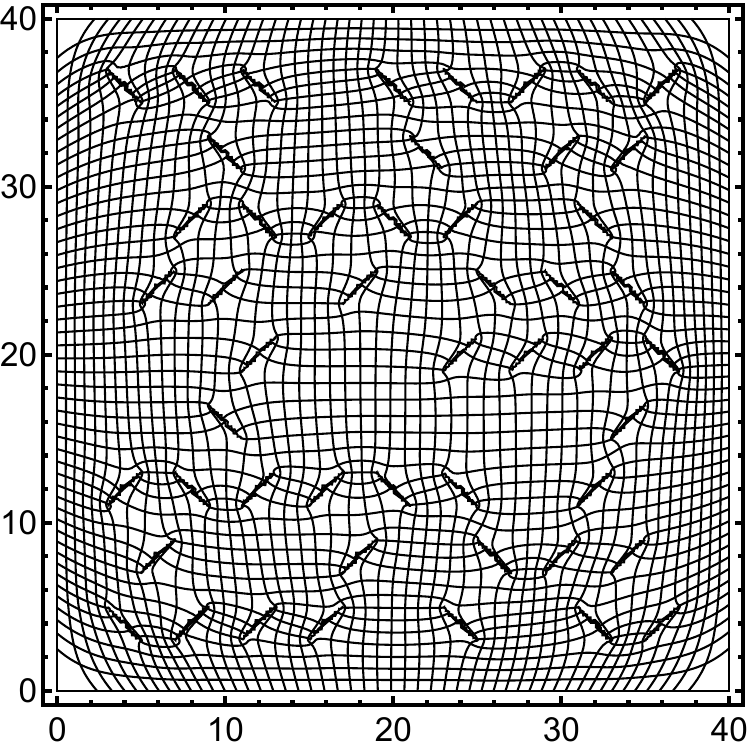}
  }\quad
  \subfloat[Conformal map on the domain.]{
  \includegraphics[height=1.5in]{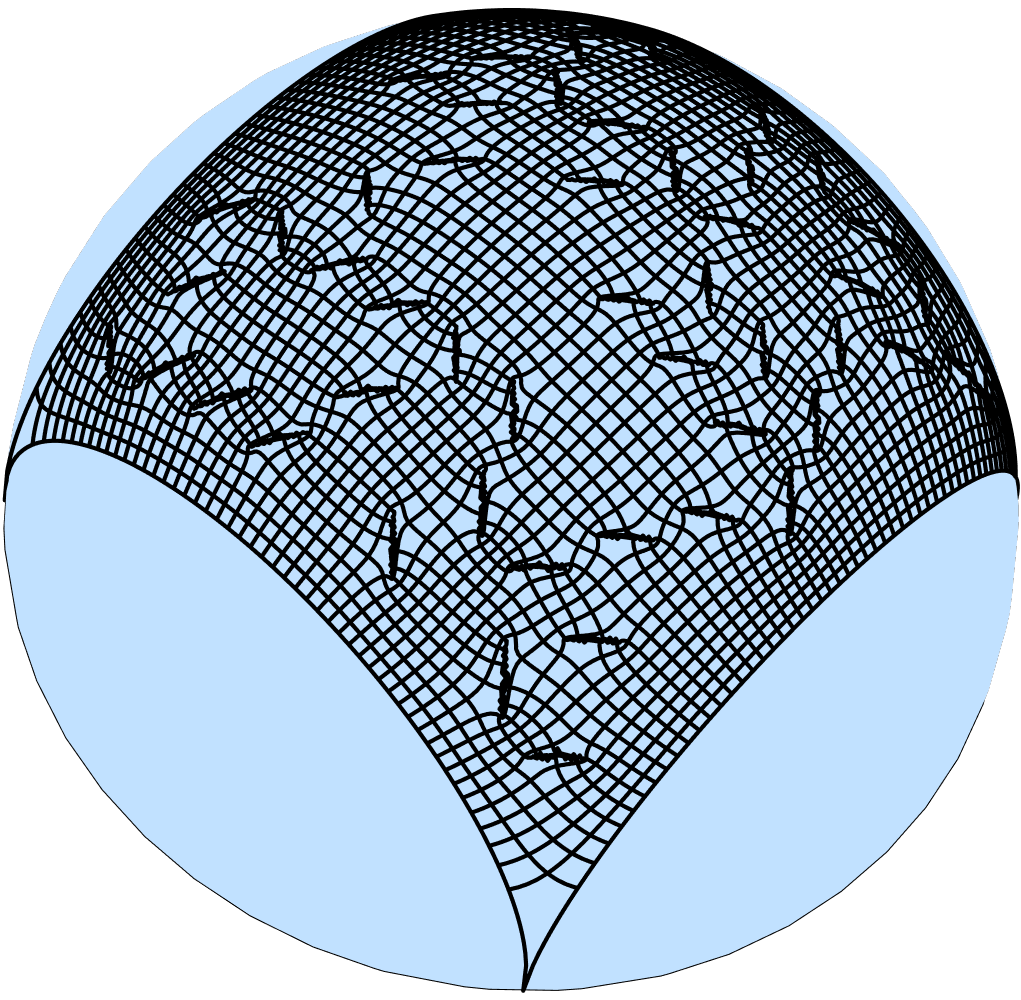}
  }
  \caption{Random Segments on Hemisphere.}\label{fig:ssegmentsA}
\end{figure}

\begin{figure}
  \centering
  \subfloat[Canonical domain.]{\label{fig:ssegmentsBa}
  \includegraphics[height=1.5in]{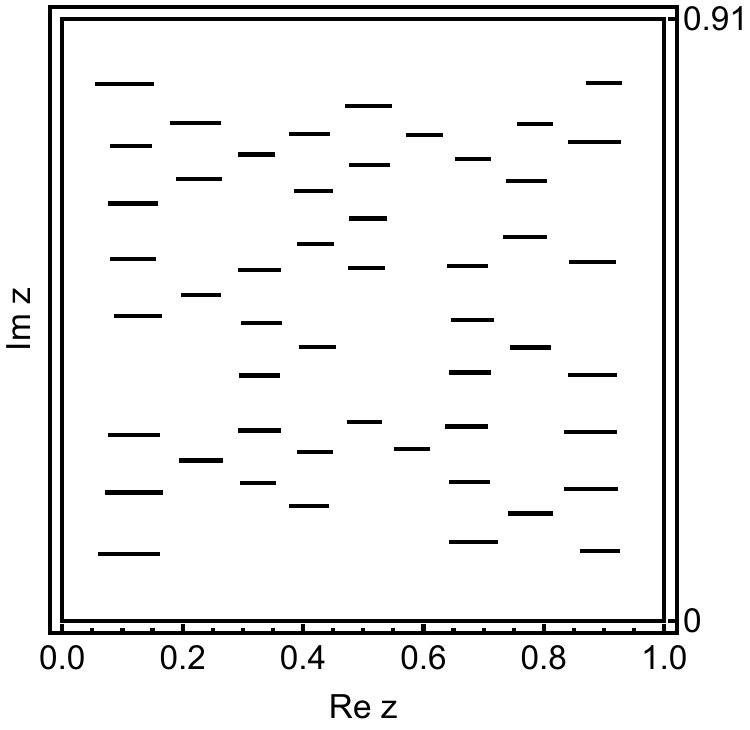}
  }\quad
  \subfloat[Comparison of the canonical domains.]{\label{fig:ssegmentsBb}
  \includegraphics[height=1.5in]{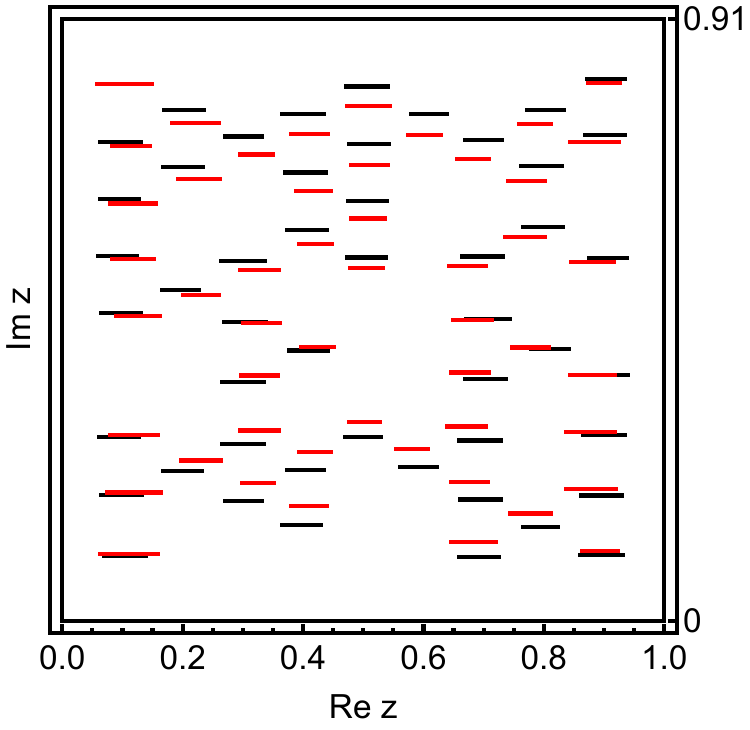}
  }\\
  \subfloat[{Convergence: Error estimates vs $N$.}]{\label{fig:ssegmentsBc}
  \includegraphics[height=1.5in]{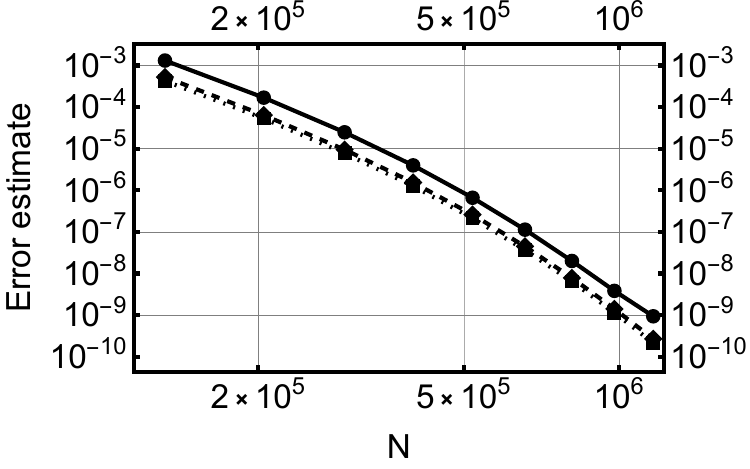}
  }
  \caption{Random Segments on Hemisphere. (a) Canonical slit domain. 
  (b) Comparison of the canonical domains: Black for planar and red for the surface.
  (c) Convergence graphs for error estimates. $N$ is the number of degrees of freedom.
    The data points represent the results on a fixed mesh with the constant polynomial order ranging from $p=2,\ldots,10$.
  Solid line: Reciprocal error; Dashed and dotted: Auxiliary space estimates for the primary and conjugate problems, respectively (loglog-plot).}\label{fig:ssegmentsB}
\end{figure}

\subsection{Comment on Computational Complexity}
The timing data over the random segment case is shown in Figure~\ref{fig:timing}.
Even in the implementation without a native sparse Cholesky decomposition, the construction of the conjugate
problem is roughly comparable to one assembly or integration of the stiffness matrix.
In \eqref{eq:conjquadratic} the latter term could be evaluated more efficiently if the Cholesky decomposition $L_{II}$ of
the matrix $A_{II} = L_{II}L_{II}^T$ were available.

\begin{figure}
  \centering
  \includegraphics[width=0.55\textwidth]{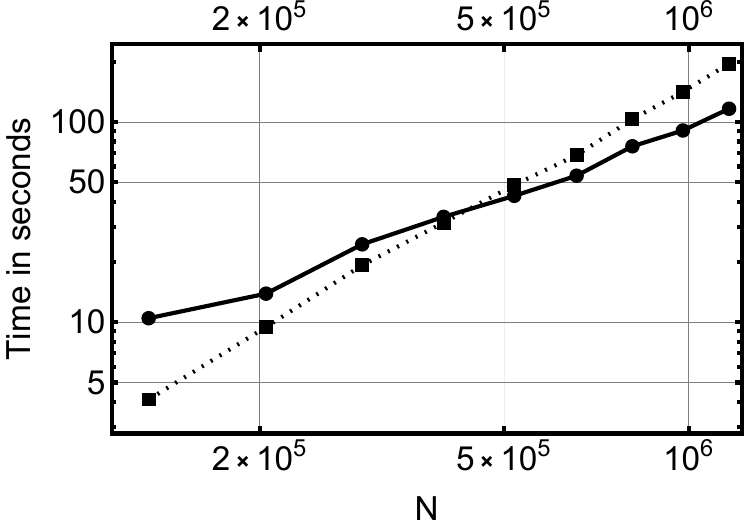}
  \caption{Timing information on Random Segments. The data points represent the results on a fixed mesh with the constant polynomial order ranging from $p=2,\ldots,10$. Solid line: Integration; Dashed: Construction of the conjugate problem (loglog-plot).
  The lack of sparse Cholesky decomposition leads to a bottleneck in \eqref{eq:conjquadratic}.}
  \label{fig:timing}
\end{figure}

\section{Application: Texture Mapping}
In this section we bring together the many aspects considered above.
We use a simplified image of a Chinese opera mask as the model\footnote{This image is derived from the original by Yudhi Sholihana, \url{https://www.vecteezy.com/free-vector/peking-opera}}. In Figure~\ref{fig:mask} this image is lifted onto a surface and
the conformal map is laid onto it.

Since in this example there are two holes (only for the eyes), the canonical domain in Figure~\ref{fig:maskConformala} is relatively simple.
However, the high accuracy of the method is reflected in the mapped image in Figure~\ref{fig:maskConformalb}. 
The intricate details of the design are clearly visible. Combined with the observation on the two canonical domains above,
this strongly suggests that these methods could be used in animation as suggested in \cite{864afe467f9949cea8cb49d061b1b5d6}.
\begin{figure}
  \centering
  \subfloat[Original image.]{\label{fig:maska}
  \includegraphics[height=2in]{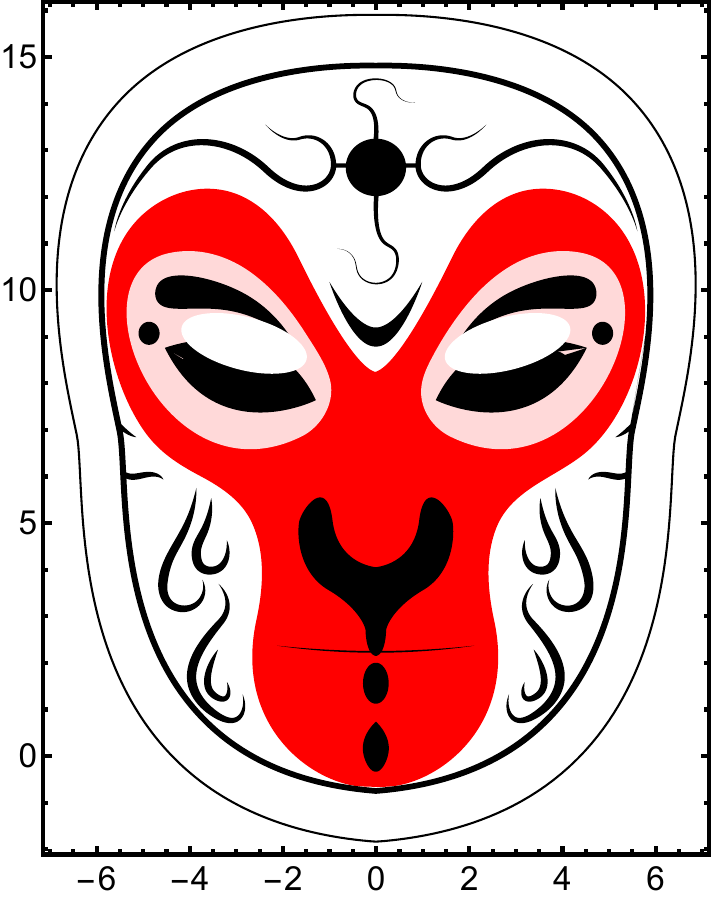}
  }\quad
  \subfloat[Conformal map on the parameter domain.]{\label{fig:maskb}
  \includegraphics[height=2in]{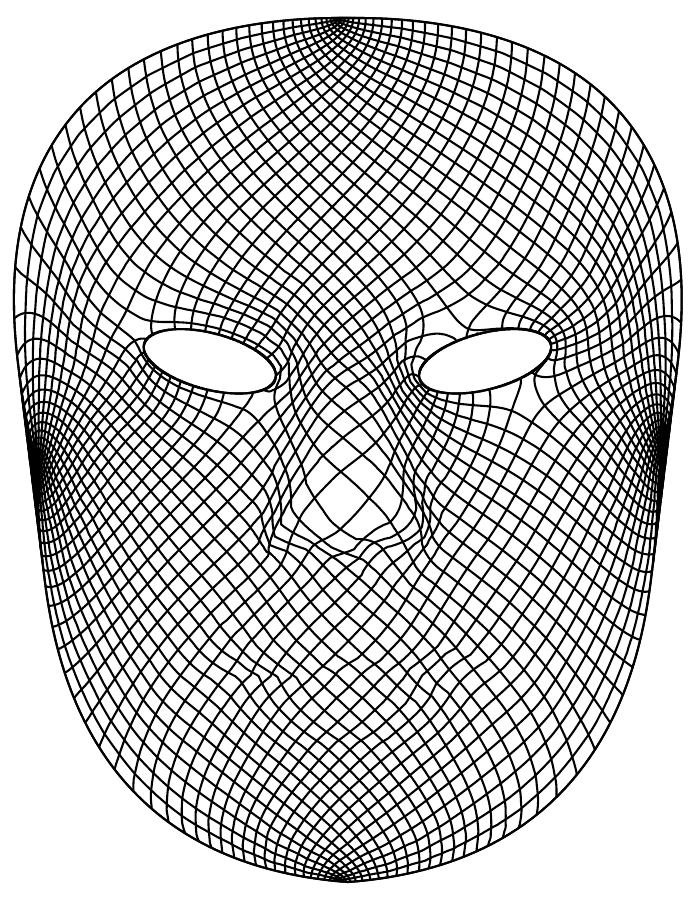}
  }\\
  \subfloat[Conformal map on the domain.]{\label{fig:maskc}
  \includegraphics[height=2.5in]{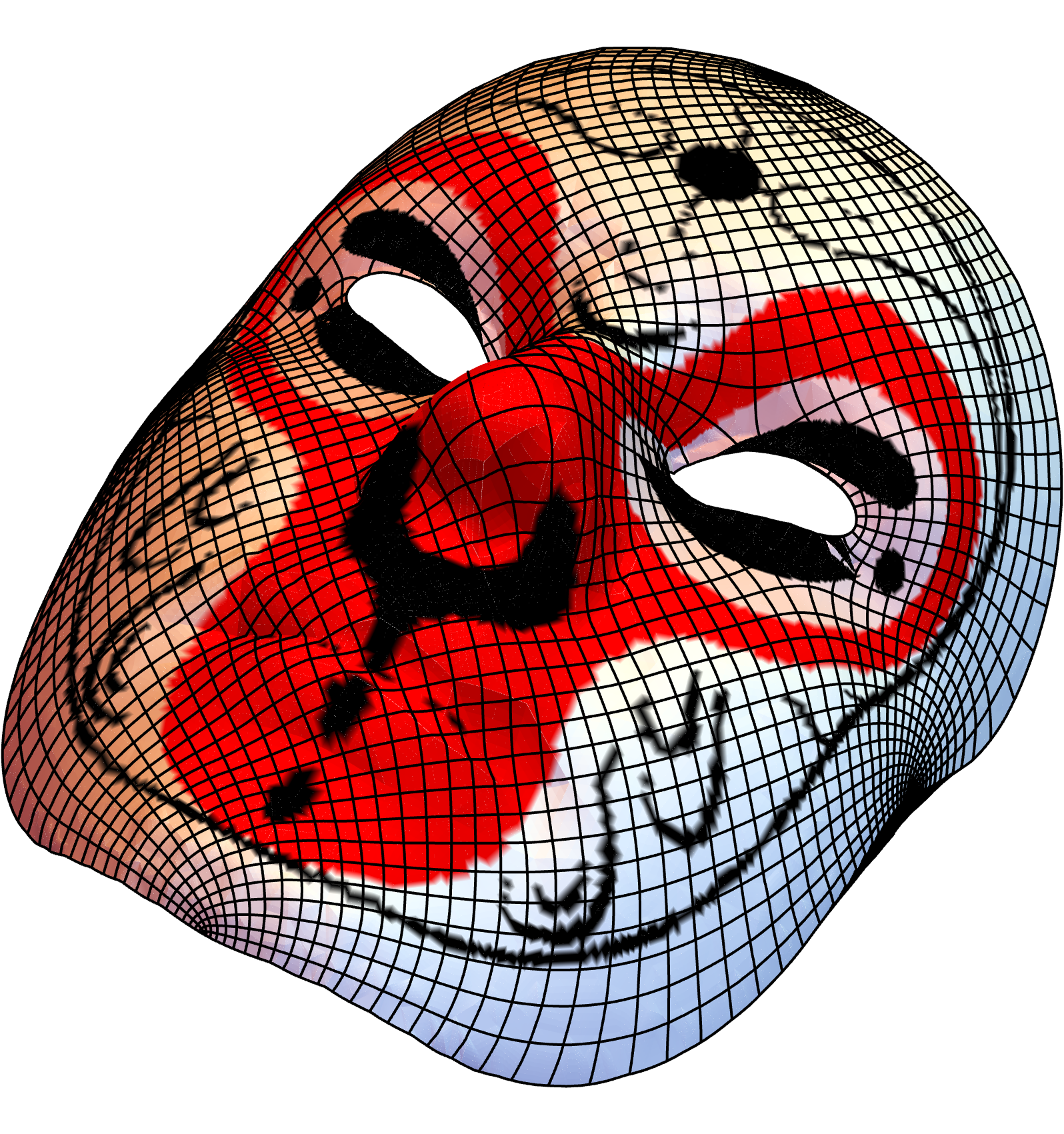}
  }
  \caption{Opera Mask. Conformal map on the domain.}\label{fig:mask}
\end{figure}

\begin{figure}
  \centering
  \subfloat[Canonical domain.]{\label{fig:maskConformala}
  \includegraphics[height=2.3in]{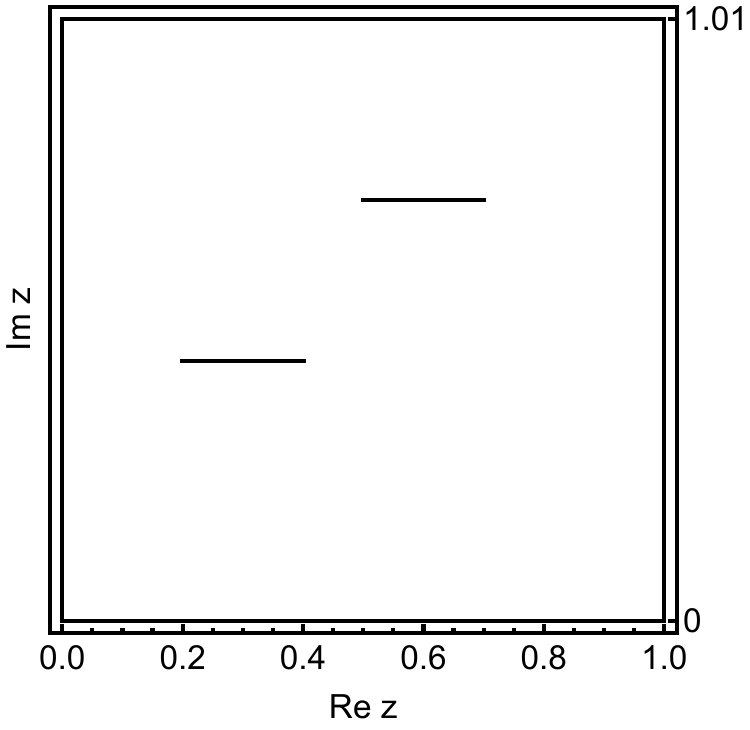}
  }\quad
  \subfloat[Image on the canonical domain.]{\label{fig:maskConformalb}
  \includegraphics[height=2.3in]{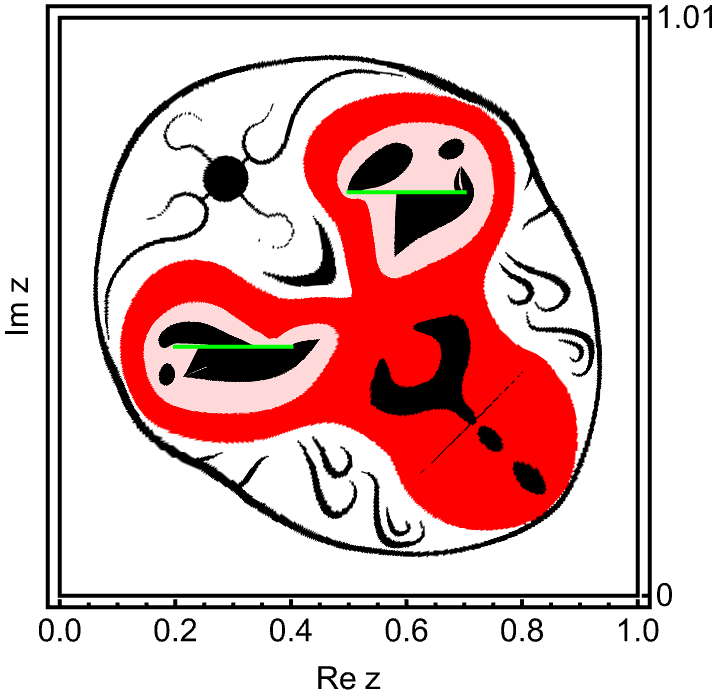}
  }
  \caption{Opera Mask. Image of the mask on the canonical domain. The slits are indicated with green lines.}\label{fig:maskConformal}
\end{figure}

\section{Polyhedral Surfaces}\label{sec:polyhedral}
We conclude the set of examples with a polyhedral closed surface \cite{beethoven3dmodel}. This is illustrated in Figure~\ref{fig:beethoven}.
The original geometric model is triangulated into 5026 3D triangles. The primary and conjugate problems are set simply by dividing the
bottom rim into four sections. The image clearly shows the crowding effect, 
the sizes of the images of the cells of the uniform grid on the canonical are different on different parts of the domain.
Since there is no global parameterization of the surface, in other words, no analytic chart
is available, the variational formulation is adjusted on every surface triangle.
\begin{figure}
  \centering
  \subfloat{
  \includegraphics[height=2.7in]{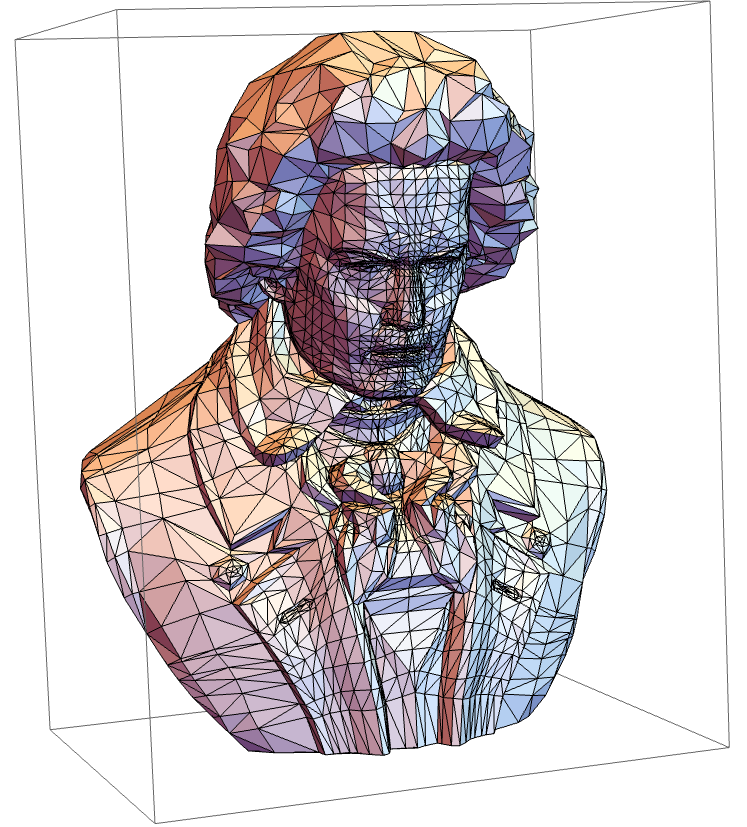}
  }
  \subfloat{
  \includegraphics[height=2.73in]{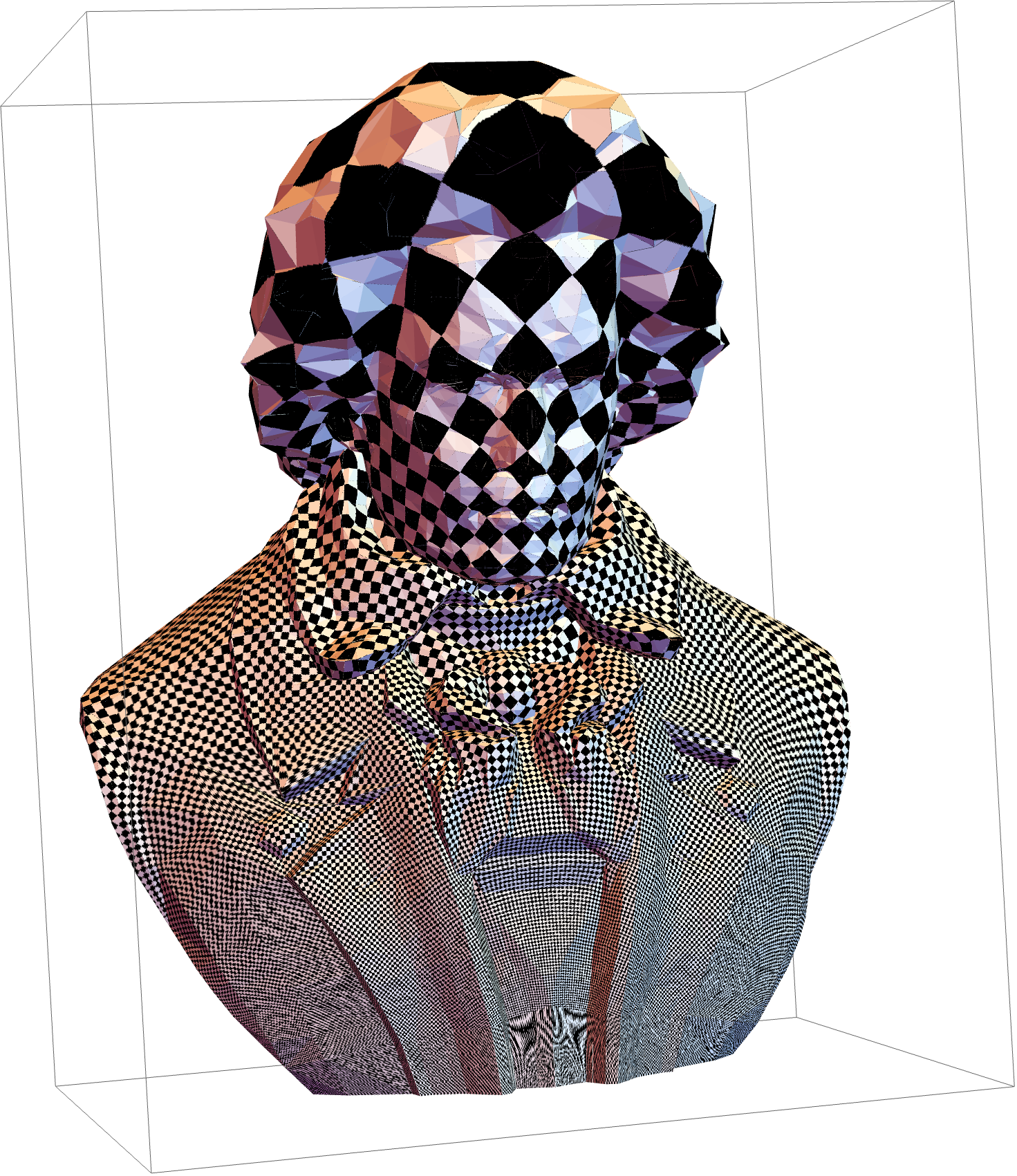}
  }
  \caption{Polyhedral surface. The geometric model on the left is triangulated into 5026 3D triangles. Using the bottom of
the statue as the boundary, a conformal map of the surface is computed.}\label{fig:beethoven}
\end{figure}

\section{Conclusions}
\label{sec:conclusions}
In this work we have generalized and refined the conjugate function method
to achieve the same level of accuracy on simply and multiply connected planar domains and Riemann surfaces.
In contrast to the Koebe iteration, the construction presented here can be viewed as direct, with predictable
a priori computational complexity. Our implementation relies on high-order finite element methods, both for
obtaining high accuracy and the construction of the conjugate problem. This does not mean that the key observation
in the minimisation process depends on one particular method for solving the underlying partial differential equations.

Extension to closed high-genus surfaces remains a future challenge. Automatic detection of charts for an atlas is a very difficult problem and
we have not made any attempts to address this question. In cases where the symmetries suggest parameterisations, 
such as the 2-torus and 3-torus considered above, the conjugate function method can be used to
compute highly accurate quasiconformal maps.

\subsection*{Acknowledgments}
We thank warmly prof. M.M.S. Nasser for kindly providing us with one of the
geometric configurations used in one of the experiments.

H. Hakula was supported by the Research Council of Finland (Flagship of Advanced Mathematics for Sensing Imaging and Modelling grant 359181).

A. Rasila and Y. Zheng were supported by NSF of Guangdong Province (No. 2024A1515010467), Li~Ka~Shing Foundation GTIIT-STU Joint Research Grant (No. 2024LKSFG06), and GTIIT Education Foundation.

\bibliographystyle{siamplain}
\bibliography{references, Yau}
\end{document}